\documentclass{amsart}
\usepackage{amsfonts,amsmath,amssymb}
\usepackage{mathrsfs,mathtools}
\usepackage{enumerate}
\usepackage{xspace}
\usepackage{esint}
\usepackage{graphicx,epstopdf}
\usepackage{textcomp}
\usepackage{calc}
\usepackage{nicefrac}
\newcommand{\sr}{{\tfrac12}}

\newcommand{\R}{\mathbb{R}}

\newcommand{\Y}{\mathpzc{Y}}

\newcommand{\N}{\mathpzc{N}}

\newcommand{\wero}{\texttt{w}}

\newcommand{\C}{\mathcal{C}}
\newcommand{\D}{\mathcal{D}}

\newcommand{\V}{\mathbb{V}}
\newcommand{\W}{\mathbb{W}}
\newcommand{\T}{\mathscr{T}}
\newcommand{\E}{\mathscr{E}}

\newcommand{\Ss}{\mathscr{S}}

\newcommand{\HL}{ \mbox{ \raisebox{6.9pt} {\tiny$\circ$} \kern-10.3pt} {H_L^1} }
\newcommand{\Wp}{ \mbox{ \raisebox{7.7pt} {\scriptsize$\circ$} \kern-10.1pt} {W^{1,p}} }
\newcommand{\Wpp}{ \mbox{ \raisebox{7.7pt} {\scriptsize$\circ$} \kern-10.1pt} {W^{1,p'}} }
\newcommand{\Sz}{ \mbox{ \raisebox{7.5pt} {\scriptsize$\circ$} \kern-10.1pt} {\Ss} }
\newcommand{\HLnew}{ \mbox{ \raisebox{7pt} {\scriptsize$\circ$} \kern-10.1pt}{H}^1_L }
\newcommand{\HLn}{{\mbox{\,\raisebox{4.5pt} {\tiny$\circ$} \kern-9.6pt}{H}^1_L  }}
\newcommand{\HLs}{{\mbox{\raisebox{8.7pt} {\scriptsize$\circ$} \kern-10.1pt}{H}^1_L  }}

\newcommand{\Tr}{\mathbb{T}}
\newcommand{\U}{\mathbb{U}}

\DeclareMathOperator*{\tr}{tr_\Omega}

\newcommand{\Laps}{(-\Delta)^s}
\newcommand{\GRAD}{\nabla}
\newcommand{\DIV}{\textrm{div}}
\newcommand{\diff}{\, \mbox{\rm d}}
\newcommand{\ie}{i.e.,\@\xspace}
\newcommand{\Hs}{\mathbb{H}^s(\Omega)}
\newcommand{\cf}{cf.\@\xspace}
\newcommand{\Wcal}{\mathcal{W}}

\newcommand{\Xcal}{\mathcal{X}}
\newcommand{\Ycal}{\mathcal{Y}}

\newcommand{\calMz}{{\mathcal{M}_{z'}}}
\newcommand{\Hsd}{\mathbb{H}^{-s}(\Omega)}
\newcommand{\ue}{\mathscr{U}}
\newcommand{\Ws}{\mathbb{H}^{1-s}(\Omega)}

\usepackage{stmaryrd}
\usepackage{bbm}
\DeclareMathAlphabet{\mathpzc}{OT1}{pzc}{m}{it}

\numberwithin{equation}{section}
\newtheorem{theorem}[equation]{Theorem}

\newtheorem{proposition}[equation]{Proposition}

\newtheorem{conjecture}[equation]{Conjecture}
\theoremstyle{definition}
\newtheorem{definition}[equation]{Definition}
\theoremstyle{definition}
\newtheorem{remark}[equation]{Remark}

\newcommand{\distb}{{\textup{\textsf{d}}_{x_0}}^{\kern-0.8em\beta}}
\newcommand{\calL}{{\mathcal L}}
\newcommand{\Wmpo}{ \mbox{ \raisebox{7.4pt} {\tiny$\circ$} \kern-10.7pt} {W_p^m} }
\newcommand{\Wonepo}{ \mbox{ \raisebox{7.4pt} {\tiny$\circ$} \kern-10.7pt} {W_p^1} }
\newcommand{\Nin}{\,{\mbox{\,\raisebox{6.0pt} {\tiny$\circ$} \kern-11.1pt}\N }}
\newcommand{\Ninn}{{\mbox{\,\raisebox{4.5pt} {\tiny$\circ$} \kern-8.8pt}\N }}

\newcommand{\osc}{{\textup{\textsf{osc}}}}
\newcommand{\bsigma}{{\boldsymbol{\sigma}}}

\begin{document}

\title[A posteriori error analysis for fractional diffusion]
{A PDE approach to fractional diffusion: \\ a posteriori error analysis}

\author[L.~Chen]{Long Chen}
\address[L.~Chen]{Department of Mathematics, University of California at Irvine, Irvine, CA 92697, USA}
\email{chenlong@math.uci.edu}
\thanks{LC has been supported by NSF grants DMS-1115961, DMS-1418732, and DOE prime award \# DE-SC0006903.}

\author[R.H.~Nochetto]{Ricardo H.~Nochetto}
\address[R.H.~Nochetto]{Department of Mathematics and Institute for Physical Science and Technology,
University of Maryland, College Park, MD 20742, USA}
\email{rhn@math.umd.edu}
\thanks{RHN has been partially supported by NSF grants DMS-1109325 and DMS-1411808.}

\author[E.~Ot\'arola]{Enrique Ot\'arola}
\address[E.~Ot\'arola]{Department of Mathematics, University of Maryland, College Park, MD 20742, USA and Department of Mathematical Sciences,
George Mason University, Fairfax, VA 22030, USA.}
\email{kike@math.umd.edu}
\thanks{EO has been partially supported by the Conicyt-Fulbright Fellowship Beca Igualdad de Oportunidades and NSF grants DMS-1109325 and DMS-1411808.}

\author[A.J.~Salgado]{Abner J.~Salgado}
\address[A.J.~Salgado]{Department of Mathematics, University of Tennessee, Knoxville, TN 37996, USA}
\email{asalgad1@utk.edu}
\thanks{AJS has been supported in part by NSF grant DMS-1418784}

                    %%  Mesh generation and refinement

\date{Version of \today.}


\begin{abstract}
We derive a computable a posteriori error estimator for the $\alpha$-harmonic extension problem, 
which localizes the fractional powers of elliptic operators supplemented with Dirichlet boundary conditions.
Our a posteriori error estimator relies on the solution of small discrete problems on 
anisotropic \emph{cylindrical stars}. It exhibits built-in flux equilibration and is 
equivalent to the energy error up to data oscillation, under suitable
assumptions. We design a simple adaptive algorithm and
present numerical experiments which reveal a competitive performance.
\end{abstract}

\maketitle

%%%%%%%%%%%%%%%%%%%%%%%%%%%%%%%%%%%%%%%%%%%%%%%%%%
\section{Introduction}
\label{sec:introduccion}
%%%%%%%%%%%%%%%%%%%%%%%%%%%%%%%%%%%%%%%%%%%%%%%%%%

The objective of this work is the derivation and analysis of a computable, 
efficient and, under certain assumptions, reliable a posteriori error estimator for problems involving 
fractional powers of the Dirichlet Laplace operator 
$\Laps$ with $s \in (0,1)$, which for convenience we will simply call the fractional Laplacian.
Let $\Omega$ be an open, connected and bounded domain of $\R^n$ ($n\ge1$) with boundary 
$\partial\Omega$, $s\in (0,1)$ and let $f:\Omega \to \R$ be given. We shall be concerned 
with the following problem: find $u$ such that
\begin{equation}
\label{fl=f_bdddom}
    \Laps u = f, \ \text{in } \Omega, \qquad u = 0, \ \text{on } \partial\Omega.
\end{equation}

One of the main difficulties in the study of problem \eqref{fl=f_bdddom} is that the fractional Laplacian is a nonlocal 
operator; see \cite{CS:11,CS:07,Landkof}. To localize it, Caffarelli and Silvestre showed in \cite{CS:07} that any power
of the fractional Laplacian in $\R^n$ can be realized as an operator that maps a Dirichlet boundary condition to a 
Neumann-type condition via an extension problem on the upper half-space $\R^{n+1}_+$. For a bounded domain $\Omega$, 
the result by Caffarelli and Silvestre has been adapted in \cite{ BCdPS:12,CDDS:11,ST:10}, thus obtaining an extension 
problem which is now posed on the semi-infinite cylinder $\C = \Omega \times (0,\infty)$. This extension is the 
following mixed boundary value problem:
\begin{equation}
\label{alpha_harm_intro}
\begin{dcases}
  \DIV\left( y^\alpha \nabla \ue \right) = 0, & \text{in } \C, \\
  \ue = 0, \quad  \text{on } \partial_L \C, &
  \quad \frac{ \partial \ue }{\partial \nu^\alpha} = d_s f, \quad \text{on } \Omega \times \{0\}, \\
\end{dcases}
\end{equation}
where $\partial_L \C= \partial \Omega \times [0,\infty)$ is
the lateral boundary of $\C$, and $d_s$ is a positive normalization constant
that depends only on $s$; see \cite{CS:11,CS:07} for details. The parameter $\alpha$ is defined as
\begin{equation}
\label{eq:defofalpha}
  \alpha = 1-2s \in (-1,1),
\end{equation}
and the so-called conormal exterior derivative of $\ue$ at $\Omega \times \{ 0 \}$ is
\begin{equation}
\label{def:lf}
\frac{\partial \ue}{\partial \nu^\alpha} = -\lim_{y \rightarrow 0^+} y^\alpha \ue_y.
\end{equation}

We will call $y$ the \emph{extended variable} and the dimension $n+1$ in $\R_+^{n+1}$ the
\emph{extended dimension} of problem \eqref{alpha_harm_intro}.
The limit in \eqref{def:lf} must be understood in the distributional sense;
see \cite{CS:11,CS:07,CDDS:11} for more details.
As noted in \cite{BCdPS:12,CS:07,CDDS:11, ST:10}, the fractional Laplacian and the Dirichlet to
Neumann operator of problem \eqref{alpha_harm_intro} are related by
\[
  d_s \Laps u = \frac{\partial \ue}{\partial \nu^\alpha } \quad \text{in } \Omega.
\]

Based on the ideas presented above, the following simple strategy to find the solution of \eqref{fl=f_bdddom}
has been proposed and analyzed by the last three authors in \cite{NOS}: given a sufficiently smooth function $f$ we solve
\eqref{alpha_harm_intro}, thus obtaining a function $\ue = \ue(x',y)$;
setting $u: x' \in \Omega \mapsto u(x') = \ue(x',0) \in \R$, we obtain 
the solution of \eqref{fl=f_bdddom}. The results of \cite{NOS} provide an a priori error analysis 
which combines asymptotic properties of Bessel functions with polynomial interpolation theory 
on weighted Sobolev spaces.
The latter is valid for tensor product elements which may be graded in $\Omega$ and exhibit a large aspect ratio in $y$ 
(anisotropy) which is necessary to fit the behavior of $\ue(x',y)$ with $x' \in \Omega$ and $y > 0$.
The resulting a priori error estimate is quasi-optimal in both order and regularity for the extended problem \eqref{alpha_harm_intro}.
These results are summarized in Section~\ref{sec:apriori}.

The main advantage of the algorithm described above, is that we are solving
the local problem \eqref{alpha_harm_intro} instead of dealing with the nonlocal operator 
$\Laps$ of problem \eqref{fl=f_bdddom}. However, this comes at the expense of incorporating one 
more dimension to the problem,
thus raising the question of computationally efficiency.
A quest for the answer has been the main drive in our recent research program and motivates 
the study of a posteriori error estimators and adaptivity. The latter is also motivated by
the fact that the a priori theory developed in \cite{NOS} requires $f \in \Ws$ and $\Omega$ convex.
If one of these conditions is violated the solution $\ue(x',y)$ may have singularities in the direction of the $x'$-variables and thus
exhibit fractional regularity. 
% The latter would not allow us to attain the almost optimal rate of convergence derived in \cite{NOS}.
As a consequence, quasi-uniform refinement of $\Omega$ would not result in an efficient solution technique and 
then an adaptive loop driven by an a posteriori error estimator 
is essential to recover optimal rates of convergence.

In this work we derive a computable, efficient and, under certain assumptions, reliable 
a posteriori error estimator and 
design an adaptive procedure to solve problem \eqref{alpha_harm_intro}. 
As the results of \cite{NOS} show, meshes must be highly anisotropic in the extended dimension $y$ 
if one intends for the method to be optimal. For this reason, it is imperative to design an a posteriori 
error estimator which is able to deal with such anisotropic behavior.
Before proceeding with our analysis, it is instructive to comment about the anisotropic a posteriori 
error estimators and analysis advocated in the literature.

A posteriori error estimators are computable quantities, \ie they may depend on the computed solution, mesh and data, 
but \emph{not} on the exact solution. They provide information about the quality of approximation 
of the numerical solution. They are problem-dependent and may be 
used to make a judicious mesh refinement in order to obtain the best possible approximation with the least 
amount of computational resources.
For isotropic discretizations, \ie meshes where the aspect 
ratio of all cells is bounded independently of the refinement level,
the theory of a posteriori error estimation is well understood. Starting with the pioneering work 
of Babu{\v{s}}ka and Rheinbolt \cite{BR78}, a great deal 
of work has been devoted to its study. We refer to \cite{AO,Verfurth} for an overview of the state of the art.
However, despite of what might be claimed in the literature, the theory of a posteriori error estimation 
on anisotropic discretizations, \ie
meshes where the cells have disparate sizes in each direction, is still in its infancy.

To the best of our knowledge the first work that attempts to deal with anisotropic
a posteriori error estimation is \cite{Siebert:96}. In this work, a residual a posteriori 
error estimator is introduced and allegedly analyzed
on anisotropic meshes. However, such analysis relies on assumptions on the exact 
and discrete solutions and on the mesh, which are neither proved nor there is a 
way to explicitly enforce them in the course of computations; see \cite[\S~6, Remark 3]{Siebert:96}.
Subsequently, in \cite{Kunert:00} the concept of \emph{matching function} is introduced 
in order to derive anisotropic a posteriori error indicators. The correct alignment of the grid 
with the exact solution is crucial to derive an upper bound for the error. Indeed, this 
upper bound involves the matching function, which depends on the error itself and then it does not provide
a \emph{real computable} quantity; see \cite[Theorem 2]{Kunert:00}. For similar works 
in this direction see \cite{KV:00,Kunert:01,Nicaise:06}. In \cite{Picasso:03}, the anisotropic
interpolation estimates derived in \cite{FP:03} are used to derive a Zienkiewicz--Zhu type 
of a posteriori error estimator. However, as properly pointed out in \cite[Proposition 2.3]{Picasso:03},
the ensuing upper bound for the error depends on the error itself, and thus, it is not \emph{computable}.

In our case, since the coefficient $y^{\alpha}$ in \eqref{alpha_harm_intro}
either degenerates $(s<1/2)$ or blows up $(s>1/2)$, the usual residual estimators
do not apply: integration by parts fails! Inspired by \cite{BabuskaMiller,MNS02}, we deal with both
the natural anisotropy of the mesh in the extended variable $y$ and the nonuniform coefficient $y^{\alpha}$, 
upon considering local problems on \emph{cylindrical stars}. The solutions
of these local problems allow us to define a computable and anisotropic a posteriori error estimator
which, under certain assumptions, is equivalent to the error up to data oscillations terms. 
In order to derive such a result,
a computationally implementable geometric condition needs to be imposed on the mesh, which does not depend on the 
exact solution of problem \eqref{alpha_harm_intro}. This approach is of value not only for 
\eqref{alpha_harm_intro}, but in general for anisotropic problems
since rigorous anisotropic a posteriori error estimators are not available in the literature. 

The outline of this paper is as follows. Section~\ref{sec:Prelim} 
sets the framework in which we will operate. 
Notation and terminology are introduced in \S~\ref{sub:notation}. 
We recall the definition of the fractional Laplacian
on a bounded domain via spectral theory in \S~\ref{sub:fractional_L}
and, in \S~\ref{sub:CaffarelliSilvestre}, we introduce function spaces that are suitable to study problems 
\eqref{fl=f_bdddom} and \eqref{alpha_harm_intro}. In Section~\ref{sec:apriori} we review the a 
priori error analysis developed in \cite{NOS}. 
The need for a new approach in a posteriori error estimation is examined in Section~\ref{sec:aposteriori_local}, where we show that the standard approaches either do not work or produce suboptimal results. This justifies the introduction of our new error estimator on \emph{cylindrical stars}.
Section~\ref{sec:aposteriori} is the core of this work and is dedicated to the development and analysis of our new error estimator. After some preliminary 
setup carried out in \S\S~\ref{subsec:preliminaries}--\ref{subsec:localspaces}, 
in \S~\ref{subsec:ideal} we introduce and analyze an ideal error estimator that, 
albeit not computable, sets the stage for \S~\ref{subsec:computable} where we devise a fully computable error estimator
and show its equivalence, under suitable assumptions, to the error up to data oscillation terms.
In Section~\ref{sec:numexp} we review the components of a standard adaptive loop and comment on some implementation details pertinent to the problem at hand. 
Finally, we present numerical experiments that illustrate and extend our theory.

%%%%%%%%%%%%%%%%%%%%%%%%%%%%%%%%%%%%%%%%%%%%%%%%%%
\section{Notation and preliminaries}
\label{sec:Prelim}
%%%%%%%%%%%%%%%%%%%%%%%%%%%%%%%%%%%%%%%%%%%%%%%%%%

%%%%%%%%%%%%%%%%%%%%%%%%%%%%%%%%%%%%%%%%%%%%%%%%%%%%%%%%%%%%%%%%%%%%%%%%%%%%%%%%%%%%%%
\subsection{Notation}
\label{sub:notation}
%%%%%%%%%%%%%%%%%%%%%%%%%%%%%%%%%%%%%%%%%%%%%%%%%%%%%%%%%%%%%%%%%%%%%%%%%%%%%%%%%%%%%%

Throughout this work $\Omega$ is an open, bounded and connected domain
of $\R^n$, $n\geq1$, with polyhedral boundary $\partial\Omega$.
We define the semi-infinite cylinder with base $\Omega$ and its lateral boundary, respectively, by
\[
 \C := \Omega \times (0,\infty), \qquad \partial_L \C  := \partial \Omega \times [0,\infty).
\]
Given $\Y>0$ we define the truncated cylinder with base $\Omega$ by
$
  \C_\Y := \Omega \times (0,\Y).
$
The lateral boundary $\partial_L\C_\Y$ is defined accordingly.

Throughout our discussion we will be dealing with objects defined in $\R^{n+1}$ 
and it will be convenient to distinguish the extended dimension.
A vector $x\in \R^{n+1}$, will be denoted by
\[
  x =  (x^1,\ldots,x^n, x^{n+1}) = (x', x^{n+1}) = (x',y),
\]
with $x^i \in \R$ for $i=1,\ldots,{n+1}$, $x' \in \R^n$ and $y\in\R$.

If $\Xcal$ and $\Ycal$ are normed vector spaces, we write $\Xcal \hookrightarrow \Ycal$
to denote that $\Xcal$ is continuously embedded in $\Ycal$. We denote by $\Xcal'$ the dual of $\Xcal$
and by $\|\cdot\|_{\Xcal}$ the norm of $\Xcal$.
The relation $a \lesssim b$ indicates that $a \leq Cb$, with a constant $C$ that does not
depend on $a$ or $b$ nor the discretization parameters. The value of $C$ might change at each occurrence.

%%%%%%%%%%%%%%%%%%%%%%%%%%%%%%%%%%%%%%%%%%%%%%%%%%%%%%%%%%%%%%%%%%%%%%%%%%%%%%%%%%%%%%
\subsection{The fractional Laplace operator}
\label{sub:fractional_L}
%%%%%%%%%%%%%%%%%%%%%%%%%%%%%%%%%%%%%%%%%%%%%%%%%%%%%%%%%%%%%%%%%%%%%%%%%%%%%%%%%%%%%%

Our definition is based on spectral theory.  
For any $f \in L^2(\Omega)$, the Lax-Milgram Lemma provides the existence and uniqueness of
$w \in H^1_0(\Omega)$ that solves
\[
 - \Delta w  = f \text{ in } \Omega, \qquad w = 0 \text{ on } \partial \Omega.
\]
The operator $(-\Delta)^{-1}: L^2(\Omega)\to L^2(\Omega)$ is compact, symmetric and positive, whence
its spectrum $\{ \lambda_k^{-1} \}_{k\in \mathbb N}$ is discrete, real, positive and accumulates at zero. 
Moreover, there exists $\{ \varphi_k \}_{k\in \mathbb N} \subset H^1_0(\Omega)$, which is
an orthonormal basis of $L^2(\Omega)$ and satisfies
\begin{equation}
  \label{eigenvalue_problem}
     - \Delta \varphi_k = \lambda_k \varphi_k  \text{ in } \Omega,
    \qquad
    \varphi_k = 0 \text{ on } \partial\Omega.
\end{equation}
Fractional powers of the Dirichlet Laplace operator can then be defined for $w \in C_0^{\infty}(\Omega)$ by
\begin{equation}
  \label{def:second_frac}
  (-\Delta)^s w  = \sum_{k=1}^\infty \lambda_k^{s} w_k \varphi_k,
\end{equation} 
where $w_k = \int_{\Omega} w \varphi_k $. By density $(-\Delta)^s$ can be extended 
to the space
\begin{equation}
\label{def:Hs}
  \Hs = \left\{ w = \sum_{k=1}^\infty w_k \varphi_k: 
  \sum_{k=1}^{\infty} \lambda_k^s w_k^2 < \infty \right\}
    =
  \begin{dcases}
    H^s(\Omega),  & s \in (0,\sr), \\
    H_{00}^{1/2}(\Omega), & s = \sr, \\
    H_0^s(\Omega), & s \in (\sr,1).
  \end{dcases}
\end{equation}
The characterization given by the second equality is shown in \cite[Chapter 1]{Lions}.
For $ s \in (0,1)$ we denote by $\Hsd$ the dual of $\Hs$.

%%%%%%%%%%%%%%%%%%%%%%%%%%%%%%%%%%%%%%%%%%%%%%%%%%%%%%%%%%%%%%%%%%%%%%%%%%%%%%%%%%%%%%
\subsection{The Caffarelli-Silvestre extension problem}
\label{sub:CaffarelliSilvestre}
%%%%%%%%%%%%%%%%%%%%%%%%%%%%%%%%%%%%%%%%%%%%%%%%%%%%%%%%%%%%%%%%%%%%%%%%%%%%%%%%%%%%%%

To exploit the Caffarelli-Silvestre result \cite{CS:07}, or its variants 
\cite{BCdPS:12,CT:10, CDDS:11}, we need to deal with a nonuniformly elliptic equation. To this end, we consider
weighted Sobolev spaces with the weight $|y|^{\alpha}$, $\alpha \in (-1,1)$.
If $D \subset \R^{n+1}$, we then define $L^2(|y|^\alpha,D)$ to be the space of all 
measurable functions defined on $D$ such that
\[
\| w \|_{L^2(|y|^{\alpha},D)}^2 = \int_{D}|y|^{\alpha} w^2 < \infty.
\]
Similarly we define the weighted Sobolev space
\[
H^1(|y|^{\alpha},D) =
  \left\{ w \in L^2(|y|^{\alpha},D): | \nabla w | \in L^2(|y|^{\alpha},D) \right\},
\]
where $\nabla w$ is the distributional gradient of $w$. We equip $H^1(|y|^{\alpha},D)$ with
the norm
\begin{equation}
\label{wH1norm}
\| w \|_{H^1(|y|^{\alpha},D)} =
\left(  \| w \|^2_{L^2(|y|^{\alpha},D)} + \| \nabla w \|^2_{L^2(|y|^{\alpha},D)} \right)^{1/2}.
\end{equation}
Since $\alpha \in (-1,1)$ we have that $|y|^\alpha$ belongs to the so-called
Muckenhoupt class $A_2(\R^{n+1})$; see \cite{FKS:82,GU,Muckenhoupt,Turesson}. This, in particular,
implies that $H^1(|y|^{\alpha},D)$ equipped with the norm \eqref{wH1norm}, is a Hilbert space
and the set $C^{\infty}(D) \cap H^1(|y|^{\alpha},D)$ is dense in $H^1(|y|^{\alpha},D)$
(cf.~\cite[Proposition 2.1.2, Corollary 2.1.6]{Turesson}, \cite{KO84} and \cite[Theorem~1]{GU}).
We recall now the definition of Muckenhoupt classes; 
see \cite{Muckenhoupt,Turesson}.

\begin{definition}[Muckenhoupt class $A_2$]
 \label{def:Muckenhoupt}
Let $\omega$ be a weight and $N \geq 1$. We say $\omega \in A_2(\R^N)$
if
\begin{equation}
  \label{A_pclass}
  C_{2,\omega} = \sup_{B} \left( \fint_{B} \omega \right)
            \left( \fint_{B} \omega^{-1} \right) < \infty,
\end{equation}
where the supremum is taken over all balls $B$ in $\R^N$.

If $\omega$ belongs to the Muckenhoupt class $A_2(\R^N)$, we say that $\omega$ is an $A_2$-weight, and
we call the constant $C_{2,\omega}$ in \eqref{A_pclass} the $A_2$-constant of $\omega$. 
\end{definition}

To study problem \eqref{alpha_harm_intro} we define the weighted Sobolev space
\begin{equation}
  \label{HL10}
  \HL(y^{\alpha},\C) = \left\{ w \in H^1(y^\alpha,\C): w = 0 \textrm{ on } \partial_L \C\right\}.
\end{equation}
As \cite[(2.21)]{NOS} shows, the following \emph{weighted Poincar\'e inequality} holds:
\begin{equation}
\label{Poincare_ineq}
  \int_{ \C }y^{\alpha}  w^2 \lesssim \int_{ \C}y^{\alpha} |\nabla w |^2,
  \quad \forall w \in \HL(y^{\alpha},\C).
\end{equation}
Then, the seminorm on $\HL(y^{\alpha},\C)$ is equivalent to the norm \eqref{wH1norm}.
For $w \in H^1(y^{\alpha},\C)$, we denote by $\tr w$ its trace onto
$\Omega \times \{ 0 \}$, and we recall that the trace operator $\tr$ satisfies,
(see \cite[Proposition 2.5]{NOS}, \cite[Proposition 2.1]{CDDS:11})
\begin{equation}
\label{Trace_estimate}
\tr \HL(y^\alpha,\C) = \Hs,
\qquad
  \|\tr w\|_{\Hs} \leq C_{\tr} \| w \|_{\HLn(y^\alpha,\C)}.
\end{equation}

Let us now describe the Caffarelli-Silvestre
result and its extension to bounded domains; see \cite{CS:07,ST:10}. 
Given $f \in \Hsd$, let $u \in \Hs$ be the solution of $(-\Delta)^s u = f$ in $\Omega$. We define 
the $\alpha$-harmonic extension of $u$ to the cylinder $\C$, as the function $\ue \in \HL(y^{\alpha},\C)$, 
solution of problem \eqref{alpha_harm_intro}, namely
\[
   (-\Delta)^s u =  d_s \frac{\partial \ue}{\partial \nu^{\alpha}} \quad \textrm{in  }  \Omega,
\quad \textrm{where  } d_s = 2^{1-2s} \frac{\Gamma(1-s)}{\Gamma(s)}.  
\]

Finally, we must mention that
\begin{equation}
\label{eq:estCtr}
  C_{\tr} \leq d_s^{-1/2}.
\end{equation}
Indeed, given $\psi \in \HL(y^\alpha, \C)$ we define $\Psi \in \HL(y^\alpha,\C)$ as the solution of
\[
  -\DIV(y^\alpha \nabla \Psi ) = 0, \text{ in } \C, \quad \Psi = 0, \text{ on } \partial_L\C, \quad \Psi = \tr \psi \text{ on } \Omega\times\{0\}.
\]
It is standard to show that $\Psi$ is the minimal norm extension of $\tr \psi$. Moreover, separation of variables 
gives $d_s \| \tr \psi \|_{\Hs}^2 = \| \nabla \Psi \|_{L^2(y^\alpha,\C)}^2$, \cite[Proposition 2.1]{CDDS:11}. Therefore
\[
  \| \tr \psi \|_{\Hs}^2 = \frac1{d_s} \| \Psi \|_{\HLn(y^\alpha,\C)}^2 \leq \frac1{d_s} \| \psi \|_{\HLn(y^\alpha,\C)}^2.
\]
Estimate \eqref{eq:estCtr} will be useful to obtain an upper bound of the error by the estimator.

%%%%%%%%%%%%%%%%%%%%%%%%%%%%%%%%%%%%%%%%%%%%%%%%%%%%%%%%%%%%%%%%%%%%%%%%%%%%%%%%%%%%%%
\section{A priori error estimates}
\label{sec:apriori}
%%%%%%%%%%%%%%%%%%%%%%%%%%%%%%%%%%%%%%%%%%%%%%%%%%%%%%%%%%%%%%%%%%%%%%%%%%%%%%%%%%%%%%
In an effort to make this contribution self-contained here we review the main results of \cite{NOS}, which deal with the a priori error analysis of discretizations of problem \eqref{fl=f_bdddom}. This will also serve to make clear the limitations of this theory, thereby justifying the quest for an a posteriori error analysis. To do so in this section, and this section only, we will assume the following regularity result, which is valid
if, for instance, the domain $\Omega$ is convex \cite{Grisvard}
\begin{equation}
\label{reg_Omega}
 \| w \|_{H^2(\Omega)} \lesssim \| \Delta_{x'} w \|_{L^2(\Omega)}, \quad \forall w \in H^2(\Omega) \cap H^1_0(\Omega). 
\end{equation}

Since $\C$ is unbounded, problem \eqref{alpha_harm_intro} cannot be directly approximated with
finite-element-like techniques. However, as \cite[Proposition 3.1]{NOS} shows, the solution $\ue$ of 
problem \eqref{alpha_harm_intro}
decays exponentially in the extended variable $y$ so that, by truncating the cylinder $\C$ to
$\C_\Y$ and setting a vanishing Dirichlet condition on the upper boundary
$y = \Y$, we only incur in an
exponentially small error in terms of $\Y$ \cite[Theorem 3.5]{NOS}.

Define
\[
  \HL(y^{\alpha},\C_\Y) = \left\{ v \in H^1(y^\alpha,\C_\Y): v = 0 \text{ on }
    \partial_L \C_\Y \cup \Omega \times \{ \Y\} \right\}.
\]
Then, the aforementioned problem reads: find $v \in \HL(y^{\alpha}, \C_\Y)$ such that
\begin{equation}
\label{alpha_harmonic_extension_weak_T}
  \int_{\C_\Y} y^\alpha \nabla v \nabla \phi
  = d_s \langle f, \tr \phi\rangle_{\Hsd \times \Hs},
\end{equation}
for all $v \in \HL(y^{\alpha},\C_\Y)$, where $\langle \cdot, \cdot \rangle_{\Hsd \times \Hs }$
denotes the duality pairing between $\Hsd$ and
$\Hs$, which is well defined as a consequence of \eqref{Trace_estimate}.

If $\ue$ and $v$ denote the solution of \eqref{alpha_harm_intro} and \eqref{alpha_harmonic_extension_weak_T}, 
respectively, then \cite[Theorem 3.5]{NOS} provides the following
exponential estimate
\begin{equation*}
% \label{le:v-v^TC}
  \| \nabla(\ue - v) \|_{L^2(y^{\alpha},\C )} \lesssim e^{-\sqrt{\lambda_1} \Y/4} \| f\|_{\Hsd},
\end{equation*}
where $\lambda_1$ denotes the first eigenvalue of the Dirichlet Laplace operator and
$\Y$ is the truncation parameter.

In order to study the finite element discretization of problem \eqref{alpha_harmonic_extension_weak_T} we 
must first understand the regularity of the solution $\ue$, since an error estimate for $v$, 
solution of \eqref{alpha_harmonic_extension_weak_T}, depends on the 
regularity of $\ue$ as well \cite[\S 4.1]{NOS}. We recall that 
\cite[Theorem 2.7]{NOS} 
reveals that the second order regularity of $\ue$ is much worse in 
the extended direction, namely
\begin{align}
    \label{reginx}
  \| \Delta_{x'} \ue\|_{L^2(y^{\alpha},\C)} + 
  \| \partial_y \nabla_{x'} \ue \|_{L^2(y^{\alpha},\C)}
  & \lesssim \| f \|_{\Ws}, \\
\label{reginy}
  \| \ue_{yy} \|_{L^2(y^{\beta},\C)} &\lesssim \| f \|_{L^2(\Omega)},
\end{align}
where $\beta > 2\alpha + 1$. This suggests that \emph{graded} meshes in the extended variable $y$ play 
a fundamental role. In fact, estimates \eqref{reginx}--\eqref{reginy} motivate the construction of a mesh over $\C_{\Y}$
as follows. We first consider a graded partition $\mathcal{I}_\Y$ of the interval $[0,\Y]$ with mesh points
\begin{equation}
\label{graded_mesh}
  y_k = \left( \frac{k}{M}\right)^{\gamma} \Y, \quad k=0,\dots,M,
\end{equation}
where $\gamma > 3/(1-\alpha)=3/(2s)$. We also consider $\T_\Omega = \{K\}$ to be a 
conforming and shape regular mesh of $\Omega$, where $K \subset \R^n$ is an element
that is isoparametrically equivalent either to the unit cube $[0,1]^n$ or the unit simplex in $\R^n$.
The collection of these triangulations $\T_{\Omega}$ is denoted by $\Tr_\Omega$. 
We construct the mesh $\T_{\Y}$ as the tensor product triangulation 
of $\T_\Omega$ and $\mathcal{I}_\Y$.
In order to obtain a global regularity assumption for $\T_{\Y}$, we assume that
there is a constant $\sigma_{\Y}$ such that
if $T_1=K_1\times I_1$ and $T_2=K_2\times I_2 \in \T_\Y$ have nonempty intersection, then
\begin{equation}
\label{shape_reg_weak}
     \frac{h_{I_1}}{h_{I_2}} \leq \sigma_{\Y},
\end{equation}
where $h_I = |I|$. It is well known that this weak regularity condition on the mesh allows for 
anisotropy in the extended variable (\cf \cite{DL:05,NOS}). The set of all triangulations of 
$\C_\Y$ that are obtained with this procedure and satisfy these conditions is denoted by $\Tr$. 
Figure~\ref{fig:ctm} shows an example of this type of meshes in three dimensions.

\begin{figure}[ht!]
\includegraphics[scale=0.44]{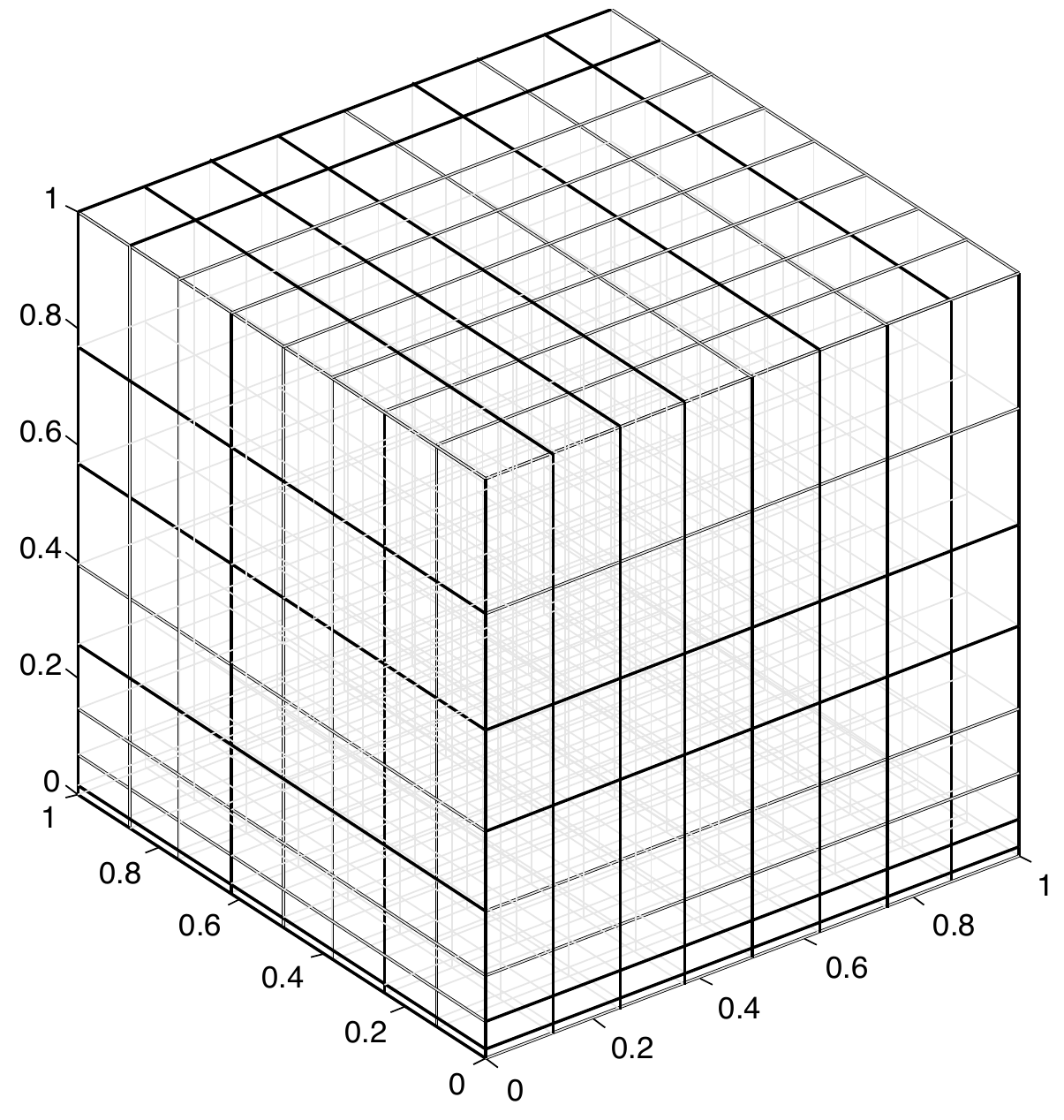}
\caption{
A three dimensional graded mesh of the cylinder
$(0,1)^2 \times (0,\Y)$ with $392$ degrees of freedom. The mesh is constructed
as a tensor product of a quasi-uniform mesh of $(0,1)^2$ with cardinality $49$
and the image of the quasi-uniform partition of the interval $(0,1)$
with cardinality $8$ under the mapping \eqref{graded_mesh}. }
\label{fig:ctm}
\end{figure}

For $\T_{\Y} \in \Tr$, we define
the finite element space 
\begin{equation}
\label{eq:FESpace}
  \V(\T_\Y) = \left\{
            W \in C^0( \overline{\C_\Y}): W|_T \in \mathcal{P}_1(K) \otimes \mathbb{P}_1(I) \ \forall T \in \T_\Y, \
            W|_{\Gamma_D} = 0
          \right\}.
\end{equation}
where $\Gamma_D = \partial_L \C_{\Y} \cup \Omega \times \{ \Y\}$ is called the
Dirichlet boundary; the space $\mathcal{P}_1(K)$ is $\mathbb{P}_1(K)$ --- the space of polynomials of total degree at most $1$, 
when the base $K$ of an element $T = K \times I$ is simplicial. If $K$ is an $n$-rectangle 
$\mathcal{P}_1(K)$ stands for $\mathbb{Q}_1(K)$ --- the space of polynomials of degree not larger than $1$ in each variable.
We also define $\U(\T_{\Omega})=\tr \V(\T_{\Y})$, \ie a 
$\mathcal{P}_1(K)$ finite element space over the mesh $\T_\Omega$.

The Galerkin approximation of \eqref{alpha_harmonic_extension_weak_T} is given by
the unique function $V_{\T_{\Y}} \in \V(\T_{\Y})$ such that
\begin{equation}
\label{harmonic_extension_weak}
  \int_{\C_\Y} y^{\alpha}\nabla V_{\T_{\Y}} \nabla W = 
d_s \langle f, \textrm{tr}_{\Omega} W \rangle_{\Hsd \times \Hs},
  \quad \forall W \in \V(\T_{\Y}).
\end{equation}
Existence and uniqueness of $V_{\T_{\Y}}$ immediately follows from $\V(\T_\Y) \subset \HL(y^{\alpha},\C_\Y)$
and the Lax-Milgram Lemma. It is trivial also to obtain a best approximation result 
\emph{\`a la} Cea. This best approximation result reduces the numerical analysis of 
problem \eqref{harmonic_extension_weak} to a question in approximation theory which in 
turn can be answered with the study of piecewise polynomial interpolation in Muckenhoupt weighted Sobolev 
spaces; see \cite{NOS,NOS2}. Exploiting the Cartesian structure of the mesh is possible to handle
anisotropy in the extended variable, construct a quasi interpolant $\Pi_{\T_\Y} : L^1(\C_\Y) \to \V(\T_\Y)$, and obtain
\begin{align*}
  \| v - \Pi_{\T_\Y} v \|_{L^2(y^\alpha,T)} & \lesssim 
    h_K  \| \nabla_{x'} v\|_{L^2(y^\alpha,S_T)} + h_I \| \partial_y v\|_{L^2(y^\alpha,S_T)}, \\
  \| \partial_{x_j}(v - \Pi_{\T_\Y} v) \|_{L^2(y^\alpha,T)} &\lesssim
    h_K  \| \nabla_{x'} \partial_{x_j} v\|_{L^2(y^\alpha,S_T)} + h_I \| \partial_y \partial_{x_j} v\|_{L^2(y^\alpha,S_T)},
\end{align*}
with $j=1,\ldots,n+1$; see \cite[Theorems 4.6--4.8]{NOS} and \cite{NOS2} for details.
However, since $\ue_{yy} \approx y^{-\alpha -1 }$ as $y \approx 0$,
we realize that $\ue \notin H^2(y^{\alpha},\C_{\Y})$
and the second estimate is not meaningful for $j=n+1$.
In view of estimate \eqref{reginy}
it is necessary to measure the regularity of $\ue_{yy}$ with a
stronger weight and thus compensate with a graded mesh in the extended
dimension. This makes anisotropic estimates essential. 

Notice that $\#\T_{\Y} = M \, \# \T_\Omega$, and that $\# \T_\Omega \approx M^n$
implies $\#\T_\Y \approx M^{n+1}$. Finally, if $\T_\Omega$ is shape regular and quasi-uniform, we have
$h_{\T_{\Omega}} \approx (\# \T_{\Omega})^{-1/n}$. All these considerations allow us to obtain the
following result; see \cite[Theorem 5.4]{NOS} and \cite[Corollary 7.11]{NOS}.

\begin{theorem}[a priori error estimate]
\label{TH:fl_error_estimates}
Let $\T_\Y \in \Tr$ be a tensor product grid, which is quasi-uniform in $\Omega$ and graded in the 
extended variable so that \eqref{graded_mesh} holds. If $\V(\T_\Y)$ is defined by \eqref{eq:FESpace} 
and $V_{\T_\Y} \in \V(\T_\Y)$ is the Galerkin approximation defined by
\eqref{harmonic_extension_weak}, then we have
\begin{equation*}
\label{optimal_rate}
  \| \ue - V_{\T_\Y} \|_{\HLn(y^\alpha,\C)} \lesssim
|\log(\# \T_{\Y})|^s(\# \T_{\Y})^{-1/(n+1)} \|f \|_{\mathbb{H}^{1-s}(\Omega)},
\end{equation*}
where $\Y \approx \log(\# \T_{\Y})$. Alternatively, if $u$ denotes the solution of \eqref{fl=f_bdddom}, then
\begin{equation*}
\| u - V_{\T_\Y}(\cdot,0) \|_{\Hs} \lesssim
|\log(\# \T_{\Y})|^s(\# \T_{\Y})^{-1/(n+1)} \|f \|_{\mathbb{H}^{1-s}(\Omega)}.
\end{equation*}
\end{theorem}

\begin{remark}[domain and data regularity]
\label{rm:dom_and_data2}
The results of Theorem~\ref{TH:fl_error_estimates} hold true
only if $f \in \mathbb{H}^{1-s}(\Omega)$ 
and the domain $\Omega$ is such that \eqref{reg_Omega} holds.
\end{remark}

%%%%%%%%%%%%%%%%%%%%%%%%%%%%%%%%%%%%%%%%%%%%%%%%%%%%%%%%%%%%%%%%%%%%%%%%%%%%%%%%%%%%%%
\section{A posteriori error estimators: the search for a new approach}
\label{sec:aposteriori_local}
%%%%%%%%%%%%%%%%%%%%%%%%%%%%%%%%%%%%%%%%%%%%%%%%%%%%%%%%%%%%%%%%%%%%%%%%%%%%%%%%%%%%%%

The function $\ue$, solution of the $\alpha$-harmonic extension problem \eqref{alpha_harm_intro},
has a singular behavior on the extended variable $y$, which is compensated by considering 
anisotropic meshes in this direction as dictated by \eqref{graded_mesh}.
However, the solution $\ue$, may also have singularities in the direction of the $x'$-variables and thus exhibit 
fractional regularity, which would not allow us to attain the almost optimal rate of convergence given 
by Theorem~\ref{optimal_rate}. In fact, as Remark~\ref{rm:dom_and_data2} indicates, it is necessary to 
require that $f \in \Ws$ and that the domain has the property 
\eqref{reg_Omega} in order to have an 
almost optimal rate of convergence. If any of these two conditions fail singularities may develop in the 
direction of the $x'$-variables, whose characterization is as yet an open problem; see \cite[\S~6.3]{NOS} for an
illustration of this situation.
The objective of this work is to derive a computable, efficient and,
under suitable assumptions, reliable
a posteriori error estimator 
for the finite element approximation of problem \eqref{alpha_harmonic_extension_weak_T} which can 
resolve such singularities via an adaptive algorithm. 

Let us begin by exploring the standard approaches advocated in the literature. We will see that 
they fail, thereby justifying the need for a new approach, 
which we will develop in Section~\ref{sec:aposteriori}.

\subsection{Residual estimators}
Simply put, the so-called residual error estimators use the strong form of the local residual as an indicator 
of the error. To obtain the strong form of the equation, integration by parts is necessary. Let us consider 
an element $T \in \T_{\Y}$ and integrate by parts the term
\[
 \int_{T} y^{\alpha} \nabla V_{\T_\Y} \cdot \nabla W = \int_{\partial T} W y^{\alpha} \nabla V_{\T_\Y} \cdot \nu 
- \int_{T} \DIV(y^{\alpha} \nabla V_{\T_\Y}) W,
\]
where $\nu$ denotes the unit outer normal to $T$. Since $\alpha \in (-1,1)$ the boundary integral is meaningless for $y = 0$. 
As we see, even the very first step (integration by parts) 
in the derivation of a residual a posteriori error estimator fails! 
At this point there is nothing left to do but to consider a different type of estimator.

\subsection{Local problems on stars over isotropic refinements}
\label{subsec:isotropic}

Inspired by \cite{BabuskaMiller,MNS02} we can construct, over shape regular meshes, a computable error 
estimator based on the solution of small discrete problems on stars. Its construction and analysis is 
similar to the developments of Section~\ref{sec:aposteriori} so we
shall not dwell on this any further.
Since we consider shape regular meshes,
such estimator is equivalent to the error up to data oscillation without any
additional conditions on the mesh, but under some suitable assumptions; see \S~5 for details.
Then, we have designed an adaptive algorithm driven by such a
posteriori error estimator on shape regular meshes \cite{MR2875241,MNS02}, and here we illustrate its performance with a simple but revealing numerical example.
We let $\Omega = (0,1)$ and $s \in (0,1)$. The right hand side is $f(x') = \pi^{2s} \sin(\pi x')$, so that
$u(x')=\sin(\pi x')$, and the solution $\ue$ to \eqref{alpha_harm_intro} is
\[
   \ue(x',y) = \frac{2^{1-s}\pi^{s}}{\Gamma(s)} \sin (\pi x') K_{s}(\pi y),
\]
where $K_s$ denotes the modified Bessel function of the second kind; see \cite[\S 2.4]{NOS} for details.
We point out that for the $\alpha$-harmonic extension we are solving a two dimensional
problem so the optimal rate of convergence in the $H^1(y^{\alpha},\C)$-seminorm that we expect is
$\mathcal{O}( \# \, \T_\Y^{-0.5})$. Figure~\ref{fig:s0.2isotropic} shows the experimental rate of 
convergence of this algorithm for the cases $s = 0.2$ and $s = 0.6$ which, as we see, is
\[
 \mathcal{O}\left(\# \T_\Y^{-s/2 } \right)
\]
and coincides with the suboptimal one obtained with 
quasi-uniform refinement; see \cite[\S~5.1]{NOS}.
These numerical experiments show that adaptive isotropic refinement cannot be optimal, thus justifying the need to 
introduce \emph{cylindrical stars} together with a new anisotropic error estimator, which will treat the $x'$-coordinates and the extended direction, $y$, separately.

\begin{figure}[ht!]
  \begin{center}
    \includegraphics[width=0.55\textwidth]{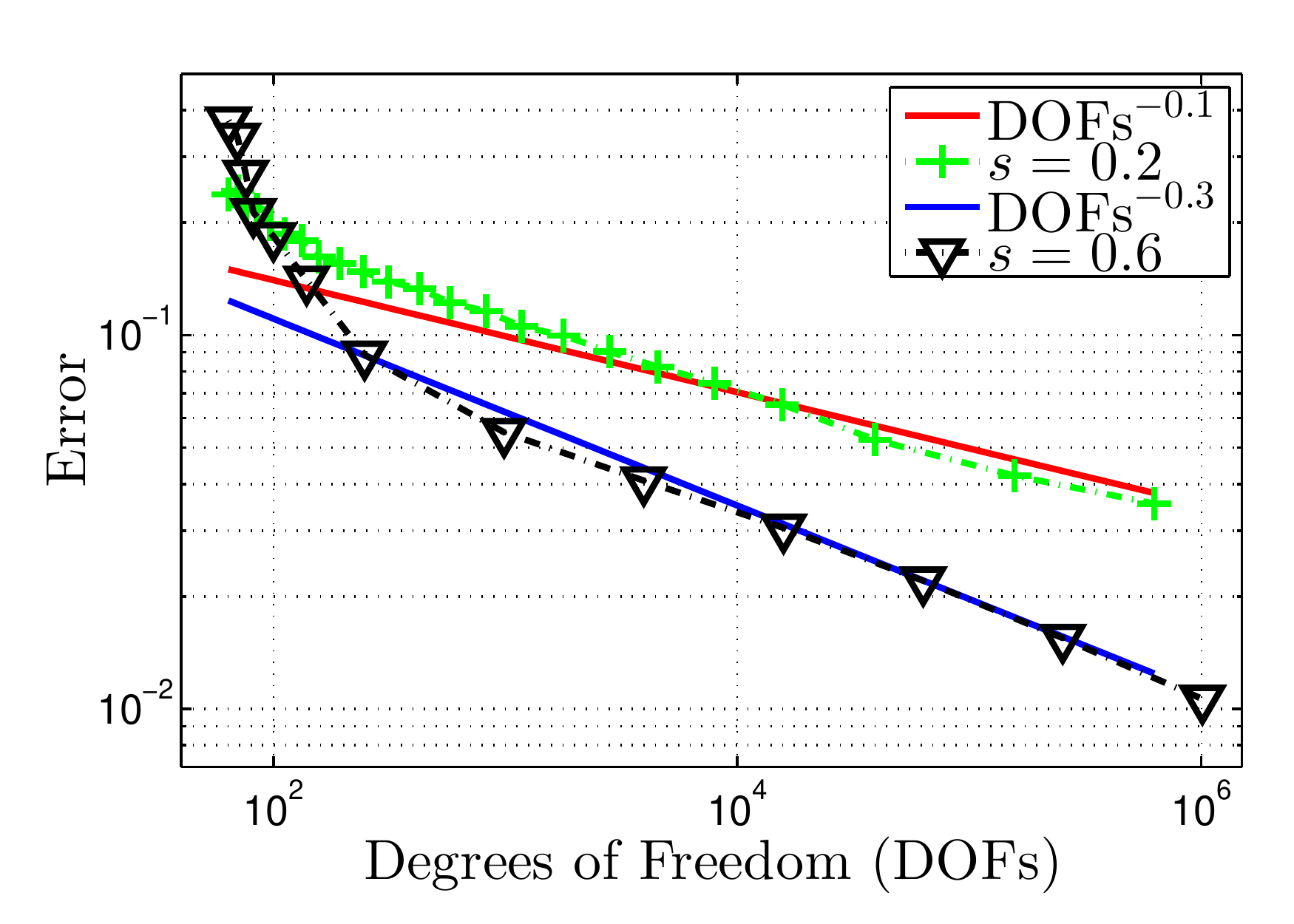}
  \end{center}
  \caption{Computational rate of convergence $\# (\T_{\Y})^{-s/2}$ for an isotropic adaptive algorithm for $n=1$, $s=0.2$ and $s=0.6$.}
\label{fig:s0.2isotropic}
\end{figure}

%%%%%%%%%%%%%%%%%%%%%%%%%%%%%%%%%%%%%%%%%%%%%%%%%%%%%%%%%%%%%%%%%%%%%%%%%%%%%%%%%%%%%%
\section{A posteriori error estimators: cylindrical stars}
\label{sec:aposteriori}
%%%%%%%%%%%%%%%%%%%%%%%%%%%%%%%%%%%%%%%%%%%%%%%%%%%%%%%%%%%%%%%%%%%%%%%%%%%%%%%%%%%%%%

It has proven rather challenging to derive and analyze a posteriori error estimators
 over a fairly general anisotropic mesh. For this reason, we introduce an \emph{implementable} 
geometric condition which will allow us to consider graded meshes in $\Omega$ in order 
to compensate for possible singularities in the $x'$-variables, while preserving 
the anisotropy in the extended direction, necessary to retain optimal orders of approximation.
We thus assume the following condition over the family of meshes $\Tr$: there exists a positive constant 
$ C_{\Tr}$ such that for every mesh $\T_{\Y} \in \Tr$
\begin{equation}
 \label{assumption}
h_{\Y} \leq C_{\Tr} \, h_{z'},
\end{equation}
for all the interior nodes $z'$ of $\T_{\Omega}$, 
where $h_{\Y}$ denotes the \emph{largest} mesh size in the $y$ direction, and 
$h_{z'} \approx |S_{z'}|^{1/n}$; see \S~\ref{subsec:preliminaries} for the precise definition of 
$h_{z'}$ and $S_{z'}$. We remark that this condition is 
satisfied in the case of quasi-uniform refinement in the variable $x'$, which is a consequence of the convexity 
of the function involved in \eqref{graded_mesh}. In fact, a simple computation shows
\begin{equation}
\label{eq:mesh_rels}
 h_{\Y}= y_{M} - y_{M-1} = \frac\Y{M^\gamma} \big( (M)^\gamma - (M-1)^\gamma \big)
 \leq \gamma \frac\Y{M},
\end{equation}
where $\gamma > 3/(1-\alpha) = 3/(2s)$. We must reiterate that this
mesh restriction is fully implementable. We refer the reader to Section~\ref{sec:numexp} for more details on this.

\begin{remark}[$s$-independent mesh grading]
We point out that the term $\gamma = \gamma(s)$ in
\eqref{eq:mesh_rels} deteriorates as $s$ becomes small
because $\gamma > 3/(2s)$. However, a 
modified mesh grading in the $y$-direction has been proposed in \cite[\S~7.3]{CNOS},
which does not change the ratio of degrees of freedom in $\Omega$ and the extended dimension by more than
a constant and provides a uniform bound with respect to $s \in (0,1)$, \ie 
$h_{\Y} \leq  C \Y / {M}$ where $C$ does not depends on $s$.
\end{remark}

%%%%%%%%%%%%%%%%%%%%%%%%%%%%%%%%%%%%%%%%%%%%%%%%%%%%%%%%%%%%%%%%%%%%%%%%%%%%%%%%%%%%%%
\subsection{Preliminaries}
\label{subsec:preliminaries}
%%%%%%%%%%%%%%%%%%%%%%%%%%%%%%%%%%%%%%%%%%%%%%%%%%%%%%%%%%%%%%%%%%%%%%%%%%%%%%%%%%%%%%

Let us begin the discussion on a posteriori error estimation with some terminology and notation. 
Given a node $z$ on the mesh $\T_{\Y}$, we exploit the tensor product structure of $\T_{\Y}$,
and  we write $z = (z',z'')$ where $z'$ and $z''$ are nodes on the meshes 
$\T_{\Omega}$ and $\mathcal{I}_{\Y}$ respectively. 
 
Given a cell $K \in \T_{\Omega}$, we denote by $\N(K)$ and $\Nin(K)$ the set of nodes and interior nodes of $K$, respectively. We set 
\[
\N(\T_{\Omega}) = \bigcup_{K \in \T_\Omega} \N(K), \qquad \Nin(\T_{\Omega}) = \bigcup_{K \in \T_\Omega} \Nin(K).
\]
Given $T \in \T_{\Y}$, we define $\N(T)$ and $\Nin(T)$ accordingly, \ie as the set of nodes and interior and Neumann nodes of $T$, respectively. Similarly, we define $\Nin(\T_{\Y})$ and $\N(\T_{\Y})$.
Any discrete function $W \in \V(\T_{\Y})$ is uniquely characterized by its  nodal values on the set 
$\Nin(\T_{\Y})$.  Moreover, the functions $\phi_z \in \V(\T_{\Y})$, $z \in \Nin(\T_{\Y})$, such that 
$\phi_z(\wero)=\delta_{z \wero}$ for all $\wero \in \N(\T_{\Y})$ are the canonical basis of $\V(\T_{\Y})$, \ie
\[
 W = \sum_{z \in \Ninn(\T_{\Y})} W(z) \phi_z.
\]
The functions $\{ \phi_z: z \in \Nin(\T_{\Y}) \}$ are the so-called \emph{shape functions} of $\V(\T_{\Y})$. 
Analogously, given a node $z' \in \Nin(\T_{\Omega})$, we also consider the discrete functions 
$ \varphi_{z'} \in \U(\T_{\Omega}) = \tr \V(\T_{\Y})$ defined by $\varphi_{z'}(\wero')=\delta_{z'\wero'}$ 
for all $\wero' \in \N(\T_{\Omega})$. 
The set $\{ \varphi_{z'}: z' \in \Nin(\T_{\Omega})\}$ is the canonical basis of $\U(\T_{\Omega})$.

The shape functions $\{ \phi_z: z \in \N(\T_{\Y}) \}$ satisfy two
properties which will prove useful in the sequel. First, we have the
so-called \emph{partition of unity property}, \ie
\begin{equation}
\label{partition}
\sum_{z \in \N(\T_{\Y})} \phi_z = 1 \quad \textrm{in} \quad \bar{\C}_{\Y}.
\end{equation}
Second, for any $z \in \Nin(\T_{\Y})$, the corresponding shape
function $\phi_{z}$ belongs to $\V(\T_\Y)$ whence we have
the so-called {\it Galerkin orthogonality}, \ie
\begin{equation}
\label{GOadap}
%   d_s\langle f, \tr \phi_{z} \rangle_{\Hsd \times \Hs} - 
\int_{\C_\Y} y^{\alpha} \nabla(v- V_{\T_{\Y}} ) \nabla \phi_{z} = 0.
\end{equation}

The partition of unity property also holds for the shape functions $\{ \varphi_{z'}: z' \in \N(\T_{\Omega}) \}$:
\[
  \sum_{z' \in \N(\T_{\Omega})} \varphi_{z'} = 1 \quad \text{in } \bar\Omega.
\]

Given $z' \in \N(\T_{\Omega})$ and the associated shape function $\varphi_{z'}$, we define the 
\emph{extended} shape function $\tilde{\varphi}_{z'}$ by
$\tilde{\varphi}_{z'}(x',y) = \varphi_{z'}(x') \mathbbm{1}_{(0,\Y)}(y)$. These functions satisfy 
the following partition of unity property:
\begin{equation}
\label{partitionx}
\sum_{z' \in \N(\T_{\Omega})} \tilde{\varphi}_{z'} = 1 \quad \text{in } \bar\C_{\Y}.
\end{equation}

Given $z' \in \N(\T_{\Omega})$, we define the \emph{star} around $z'$ as
\[
  S_{z'} = \bigcup_{K \ni z'} K \subset \Omega,
\]
and the \emph{cylindrical star} around $z'$ as
\[
  \C_{z'} := \bigcup\left\{ T \in \T_\Y : T = K \times I,\ K \ni z'  \right\}= S_{z'} \times (0,\Y) \subset \C_{\Y}.
\]
Given an element $K \in \T_{\Omega}$ we define its \emph{patch} as 
$
  S_K := \bigcup_{z' \in K} S_{z'}.
$
For $T \in \T_\Y$ its patch $S_T$ is defined similarly. Given $z' \in \N(\T_{\Omega})$
we define its \emph{cylindrical patch} as
\[
\D_{z'} := \bigcup  \left\{ \C_{\wero'}: \wero' \in S_{z'} \right\} \subset \C_{\Y}.
\]
For each $z' \in \N(\T_{\Omega})$ we set $h_{z'} := \min\{h_{K}: K \ni z' \}$. 

%%%%%%%%%%%%%%%%%%%%%%%%%%%%%%%%%%%%%%%%%%%%%%%%%%%%%%%%%%%%%%%%%%%%%%%%%%%%%%%%%%%%%%
\subsection{Local weighted Sobolev spaces}
\label{subsec:localspaces}
%%%%%%%%%%%%%%%%%%%%%%%%%%%%%%%%%%%%%%%%%%%%%%%%%%%%%%%%%%%%%%%%%%%%%%%%%%%%%%%%%%%%%%

In order to define the local a posteriori error estimators we first need to 
define some local weighted Sobolev spaces. 

\begin{definition}[local spaces]
Given $z' \in \N(\T_\Omega)$ and its associated cylindrical star $\C_{z'}$, we define
\[
\W(\C_{z'}) = \left \{ w \in H^1(y^{\alpha},\C_{z'} ): w = 0 \textrm{ on } \partial \C_{z'} 
\setminus \Omega \times \{ 0\} \right \}.
\]
\end{definition}

The space $\W(\C_{z'})$ defined above is Hilbert due to the fact that the weight $|y|^\alpha$ belongs to the 
class $A_2(\R^{n+1})$; see Definition~\ref{def:Muckenhoupt}. Moreover, as the following result shows, a weighted Poincar\'e-type inequality holds and, consequently,
the semi-norm $\interleave w \interleave_{\C_{z'}} = \| \nabla w \|_{L^2(y^{\alpha},\C_{z'})}$ defines a norm on $\W(\C_{z'})$; see also \cite[\S~2.3]{NOS}.

\begin{proposition}[weighted Poincar\'e inequality]
\label{pro:poincare}
Let $z' \in \N(\T_{\Omega})$. If the function $w \in \W(\C_{z'})$, 
then we have
\begin{equation}
\label{eq:Poincare}
  \| w \|_{L^2(y^{\alpha},\C_{z'})} \lesssim \Y \interleave w \interleave_{\C_{z'}}.
\end{equation}
\end{proposition}
\begin{proof}
By density \cite[Corollary 2.1.6]{Turesson}, it suffices to reduce the considerations to a smooth function $w$. Given 
$x' \in S_{z'}$, we have that $w(x',\Y) = 0$ so that
\[
  w(x',y) = -\int_{y}^{\Y} \partial_y w(x',\xi) \diff \xi.
\]
Multiplying the expression above by $|y|^{\alpha}$, integrating over $\C_{z'}$, 
and using the Cauchy Schwarz inequality, we arrive at
\begin{align*}
  \int_{\C_{z'}} |y|^{\alpha}| w(x',y) |^2 \diff x' \diff y 
  & \leq  
  \int_{\C_{z'}} |y|^{\alpha} \left( \int_{0}^{\Y} |\xi|^\alpha | \partial_y w(x',\xi) |^2 \diff \xi \int_{0}^{\Y} |\xi|^{-\alpha} \diff \xi \right)  \diff x' \diff y
  \\
  & =   \int_{0}^{\Y} |y|^{\alpha}\diff y   \int_{0}^{\Y} |\xi|^{-\alpha}\diff \xi  
  \int_{\C_{z'}} |\xi|^\alpha | \partial_y w(x',\xi) |^2 \diff x' \diff \xi
  \\
  & \leq   C_{2,|y|^{\alpha}} \Y^2  \int_{\C_{z'}} |y|^\alpha | \partial_y w(x',y) |^2 \diff x' \diff y,
 \end{align*}
where in the third inequality we used that $y^{\alpha} \in A_2(\R^{n+1})$. In conclusion,
\[
  \| w \|_{L^2(y^{\alpha},\C_{z'})} \lesssim \Y  \| \partial_y w \|_{L^2(y^{\alpha},\C_{z'})}, 
\]
which is \eqref{eq:Poincare}.
\end{proof}

\begin{remark}[anisotropic weighted Poincar\'e inequality]
\label{re:poincare}
Let $z' \in \N(\T_{\Omega})$. If $w \in \W(\C_{z'})$, then by extending the one-dimensional argument in the proof of Proposition~\ref{pro:poincare}
to a $n$-dimensional setting, we can also derive
\[
  \| w \|_{L^2(y^{\alpha},\C_{z'})} \lesssim h_{z'} \| \nabla_{x'} w \|_{L^2(y^{\alpha},\C_{z'})}.
\]
\end{remark}

%%%%%%%%%%%%%%%%%%%%%%%%%%%%%%%%%%%%%%%%%%%%%%%%%%%%%%%%%%%%%%%%%%%%%%%%%%%%%%%%%%%%%%
\subsection{An ideal a posteriori error estimator}
\label{subsec:ideal}
%%%%%%%%%%%%%%%%%%%%%%%%%%%%%%%%%%%%%%%%%%%%%%%%%%%%%%%%%%%%%%%%%%%%%%%%%%%%%%%%%%%%%%
Here we define an \emph{ideal} a posteriori error estimator on anisotropic meshes
which is \emph{not computable}. However, it provides the intuition required to define a 
discrete and computable error indicator, as explained in \S~\ref{subsec:computable}. 
On the basis of assumption \eqref{assumption}, we prove that this 
ideal error estimator is equivalent to the error without any oscillation terms.

Inspired by \cite{BabuskaMiller,CF:00,MNS02} we define $\zeta_{z'} \in \W(\C_{z'})$ to be the solution of
    \begin{equation}
    \label{local_problem}
    \int_{\C_{z'}} y^{\alpha} \nabla\zeta_{z'} \nabla \psi 
    =
    d_s \langle f, \tr \psi  \rangle_{\Hsd \times \Hs} 
     -  \int_{\C_{z'}} y^{\alpha} \nabla V_{\T_{\Y}} \nabla \psi ,
    \end{equation}
for all $\psi \in \W(\C_{z'})$. The existence and uniqueness of $\zeta_{z'} \in \W(\C_{z'})$
is guaranteed by the Lax--Milgram Lemma and the weighted Poincar\'e inequality of Proposition~\ref{pro:poincare}.
The continuity of the right hand side of \eqref{local_problem}, as a linear functional in 
$\W(\C_{z'})$, follows from \eqref{Trace_estimate} and
the Cauchy-Schwarz inequality. 
We then define the global error estimator
  \begin{equation}
  \label{noncompestimator}
  \tilde \E_{\T_{\Y}} = \left( \sum_{z' \in \N(\T_\Omega) } \tilde \E_{z'}^2 \right)^{1/2},
  \end{equation}
in terms of the local error indicators
  \begin{equation}
  \label{noncomplocal}
  \tilde \E_{z'} = \interleave \, \zeta_{z'} \! \interleave_{\C_{z'}}.
  \end{equation}
The properties of this ideal estimator are as follows.

\begin{proposition}[ideal estimator]
\label{pro:ideal_estimator}
Let $v \in \HL(y^\alpha,\C_\Y)$ and $V_{\T_\Y} \in \V(\T_\Y)$ solve \eqref{alpha_harmonic_extension_weak_T} 
and \eqref{harmonic_extension_weak}, respectively. Then,
the ideal estimator $\tilde{\E}_{\T_\Y}$, defined in \eqref{noncompestimator}--\eqref{noncomplocal}, satisfies
\begin{equation}
  \label{ideal_bounds1}
  \|\nabla(v - V_{\T_{\Y}}) \|_{L^2(y^{\alpha},\C_\Y)} \lesssim \tilde \E_{\T_{\Y}},
\end{equation}
and for all $z' \in \N(\T_{\Omega})$
\begin{equation}
\label{ideal_bounds2}
\tilde \E_{z'} \leq \|\nabla(v - V_{\T_{\Y}}) \|_{L^2(y^{\alpha},\C_{z'})}.
\end{equation}
\end{proposition}
\begin{proof}
If $e_{\T_{\Y}} := v - V_{\T_{\Y}}$ denotes the error, then for any $w \in \HL(y^{\alpha},\C_\Y)$ we have 
\begin{align*}
 \int_{\C_\Y} y^{\alpha} & \nabla e_{\T_\Y} \nabla w = 
      d_s \langle f, \tr w \rangle_{\Hsd \times \Hs}
                      -  \int_{\C_\Y} y^{\alpha} \nabla V_{\T_{\Y}} \nabla w
      \\
      =  & \quad d_s \langle f, \tr (w -W)
      \rangle_{\Hsd \times \Hs} - \int_{\C_{\Y}} y^{\alpha} \nabla V_{\T_{\Y}} \nabla \left(w-W\right)
      \\
      =  & \sum_{z' \in \N(\T_\Omega)} \! d_s \langle f, \tr [(w-W) \tilde \varphi_{z'}]  
      \rangle_{\Hsd \times \Hs}
                    - \int_{\C_{z'}} y^{\alpha} \nabla V_{\T_{\Y}} \nabla [(w -W)\tilde \varphi_{z'}]
\end{align*}
for any $W \in \V(\T_{\Y})$, where to derive the expression above we have used 
Galerkin orthogonality \eqref{GOadap} and the partition of unity property \eqref{partitionx}.

Notice that for each ${z'} \in \N(\T_{\Omega})$ the function $(w - W) \tilde \varphi_{z'} \in \W(\C_{z'})$. Indeed, if $z'$ is an interior node, 
\begin{equation}
\label{vanishing}
  (w - W) \tilde \varphi_{z'}|_{\partial \C_{z'} \setminus \Omega \times \{0\} } = 0  
\end{equation}
because of the vanishing property of the shape function $\varphi_{z'}$ on $\partial S_{z'}$ together with the fact that
$w = W = 0$ on $\Omega \times \{ \Y \}$; otherwise, if $z'$ is a
Dirichlet node then $w = W = 0$ on $\partial\C_{z'}$, and we get \eqref{vanishing}.

Set $W = \Pi_{\T_{\Y}}w$, where $\Pi_{\T_\Y}$ denotes the
quasi-interpolation operator introduced in \cite[\S~4]{NOS}; see also \cite{NOS2}. This yields a bound on
$\interleave (w - W) \tilde \varphi_{z'} \interleave_{\C_{z'}}$ as follows:
\begin{align*}
  \interleave (w - \Pi_{\T_\Y}w) \tilde \varphi_{z'} \interleave_{\C_{z'}}^2 & \lesssim 
  \int_{\C_{z'}} y^{\alpha}|\nabla (w -\Pi_{\T_{\Y}} w )|^2 \tilde
  \varphi_{z'}^2 \\
  & + \int_{\C_{z'}}  y^{\alpha}|w -\Pi_{\T_{\Y}} w|^2
    |\nabla_{x'} \tilde \varphi_{z'}|^2 
    \lesssim \interleave w \interleave_{\D_{z'}}^2.
\end{align*}
To bound the first term above we use the local stability of $\Pi_{\T_{\Y}}$ \cite[Theorems 4.7 and 4.8]{NOS} together with
the fact that $0\leq \tilde \varphi_{z'} \leq 1$ for all $x \in
\C_{\Y}$. For the second term we resort to
the local approximation properties of $\Pi_{\T_{\Y}}$ \cite[Theorems 4.7 and 4.8]{NOS}
\begin{equation}
  \begin{aligned}
    \int_{\C_{z'}}  y^{\alpha}|w-\Pi_{\T_{\Y}} w|^2 |\nabla_{x'} \tilde \varphi_{z'}|^2 & \lesssim 
    \frac{1}{h_{z'}^2}\left( h_{z'}^2 \| \nabla_{x'} w \|_{L^2(\D_{z'},y^{\alpha})}^2 \right. \\
      &+ \left. h_{\Y}^2 \| \partial_y w \|_{L^2(\D_{z'},y^{\alpha})}^2 \right) 
       \lesssim \interleave w \interleave_{\D_{z'}}^2,
  \end{aligned}
\label{key_step}
\end{equation}
where we used that $|\nabla_{x'} \tilde \varphi_{z'}|  = |\nabla_{x'} \varphi_{z'}|\lesssim h_{z'}^{-1}$
together with \eqref{assumption}.

Set $\psi_{z'} = (w - \Pi_{\T_{\Y}} w) \tilde \varphi_{z'} \in \W(\C_{z'})$ as 
test function in \eqref{local_problem} to obtain
\begin{align*}
  \int_{\C_\Y} y^{\alpha} \nabla e_{\T_\Y} \nabla w & = 
  \sum_{z' \in \N(\T_\Omega)} \int_{\C_{z'}} y^{\alpha} \nabla \zeta_{z'} \nabla \psi_{z'} 
  \lesssim \sum_{z'\in \N(\T_\Omega)} \interleave \zeta_{z'} \! \! \interleave_{\C_{z'}} \interleave w \interleave_{\D_{z'}} \\
  & \lesssim  \left(\sum_{z' \in \N(\T_{\Omega})} \interleave \zeta_{z'} \interleave^2_{\C_{z'}} \right)^{1/2} 
 \| \nabla w \|_{  L^2(y^{\alpha},\C_\Y) }
  = \tilde \E_{\T_{\Y}} \| \nabla w \|_{  L^2(y^{\alpha},\C_\Y)},
\end{align*}
where we used that 
$
\interleave \psi_{z'} \interleave_{\C_{z'}} \lesssim \interleave w \interleave_{\D_{z'}}
$
and the finite overlapping property of the stars $S_{z'}$. 
Since $w$ is arbitrary, we set $w = e_{\T_\Y}$ and obtain \eqref{ideal_bounds1}.

Finally, inequality \eqref{ideal_bounds2} is immediate:
\[
  \tilde \E_{z'}^2 = \interleave \zeta_{z'} \interleave^2_{\C_{z'}} \!= \int_{\C_{z'}} y^{\alpha} \nabla 
  \zeta_{z'} \nabla \zeta_{z'} =  
  \int_{\C_{z'}} y^{\alpha} \nabla e_{\T_\Y} \nabla \zeta_{z'} \leq \| \nabla e_{\T_\Y} \|_{ L^2(y^{\alpha},\C_{z'}) }   
  \interleave \zeta_{z'}  \interleave_{\C_{z'}}.
\]
This concludes the proof.
\end{proof}

\begin{remark}[anisotropic meshes]
Examining the proof of Proposition \ref{pro:ideal_estimator}, we realize
that a critical part of \eqref{key_step} consists of 
the application of inequality \eqref{assumption}, 
which we recall reads: $h_{\Y} \leq C_{\Tr} \, h_{z'}$ for all $z' \in \N(\T_{\Omega})$. Therefore,
Proposition \ref{pro:ideal_estimator} shows how the resolution of local problems on 
cylindrical stars allows for anisotropic meshes 
on the extended variable $y$ and graded meshes in $\Omega$.
The latter enables us to compensate possible singularities in the $x'$-variables.
\end{remark}

\begin{remark}[relaxing the mesh condition]\label{relax_mesh}
Owing to the anisotropy of the mesh, condition \eqref{assumption}
can be violated only near the top of the cylinder. 
Near the bottom of the cylinder, the size of the elements will be much smaller
than $h_{z'}$. If one could prove that the error $v- V_{\T_\Y}$ decays exponentially 
(as the exact solution $v$ does) an examination of the proof of Proposition \ref{pro:ideal_estimator} reveals that condition \eqref{assumption} can be removed. 
Proving this decay, however, requires local pointwise error estimates which are not availabe 
and are currently under investigation.
\end{remark}
%%%%%%%%%%%%%%%%%%%%%%%%%%%%%%%%%%%%%%%%%%%%%%%%%%%%%%%%%%%%%%%%%%%%%%%%%%%%%%%%%%%%%%
\subsection{A computable a posteriori error estimator}
\label{subsec:computable}
%%%%%%%%%%%%%%%%%%%%%%%%%%%%%%%%%%%%%%%%%%%%%%%%%%%%%%%%%%%%%%%%%%%%%%%%%%%%%%%%%%%%%%

Although Proposition~\ref{pro:ideal_estimator} shows that the error estimator 
$\tilde \E_{\T_\Y}$ is ideal, it has an insurmountable drawback: 
for each node $z' \in \N(\T_\Omega)$, it requires knowledge of the exact 
solution $\zeta_{z'}$ to the local problem \eqref{local_problem} 
which lies in the infinite dimensional space $\W(\C_{z'})$. This makes this estimator not 
computable. However, it provides intuition and establishes the basis to define a 
discrete and computable error estimator. To achieve this, let us now define local 
discrete spaces and local computable error indicators, 
on the basis of which we will construct our global error estimator.

\begin{definition}[discrete local spaces]
\label{def:discrete_spaces}
For $z' \in \N(\T_\Omega)$, define the discrete space
\begin{align*}
  \mathcal{W}(\C_{z'}) &= 
    \left\{
      W \in C^0( \bar{\C}_\Y): W|_T \in \mathcal{P}_2(K) \otimes \mathbb{P}_2(I) \ \forall T = K \times I \in \C_{z'}, \right. \\
    &\left.  W|_{\partial \C_{z'} \setminus \Omega \times \{ 0\} } = 0
    \right\}.
\end{align*}
where, if $K$ is a quadrilateral, $\mathcal{P}_2(K)$ stands for $\mathbb{Q}_2(K)$ --- the space of polynomials of degree not larger than $2$ in each variable.
If $K$ is a simplex, $\mathcal{P}_2(K)$ corresponds to 
$\mathbb{P}_2(K) \oplus \mathbb{B}(K)$ where where $\mathbb{P}_2(K)$ stands for 
the space of polynomials of total degree at most $2$, and $\mathbb{B}(K)$ is the space spanned by a local cubic bubble function.
\end{definition}

We then define the discrete local problems: for each cylindrical star $\C_{z'}$ we define 
$\eta_{z'} \in \mathcal{W}(\C_{z'})$ to be the solution of
\begin{equation}
\label{discrete_local_problem}
  \int_{\C_{z'}} y^{\alpha} \nabla\eta_{z'} \nabla W =
  d_s \langle f, \tr W \rangle_{\Hsd \times \Hs}
   -  \int_{\C_{z'}} y^{\alpha} \nabla V_{\T_{\Y}} \nabla W,
\end{equation}
for all $W \in \mathcal{W}(\C_{z'})$. We also define the global error estimator
  \begin{equation}
  \label{comp_global_estimator}
  \E_{\T_{\Y}} = \left( \sum_{z' \in \N(\T_{\Omega}) } \E_{z'}^2 \right)^{\sr},
  \end{equation}
in terms of the local error indicators
  \begin{equation}
  \label{comp_local_estimator}
  \E_{z'} = \interleave \eta_{z'} \! \!  \interleave_{\C_{z'}}.
  \end{equation}

We next explore the connection between the estimator
  \eqref{comp_global_estimator} and the error. We 
first prove a local lower bound for the error without any oscillation term and free of any constant.

\begin{theorem}[localized lower bound]
\label{th:dis_lower_bound}
Let $v \in \HL(y^\alpha,\C_\Y)$ and $V_{\T_\Y} \in \V(\T_\Y)$ solve 
\eqref{alpha_harmonic_extension_weak_T} and \eqref{harmonic_extension_weak} respectively. 
Then, for any $z' \in \N(\T_{\Omega})$, we have 
\begin{equation}
\label{dis_lower_bound}
\E_{z'} \leq\| \nabla(v-V_{\T_{\Y}})\|_{L^2(y^{\alpha},\C_{z'})}.
\end{equation}
\end{theorem}
\begin{proof}
The proof repeats the arguments employed to obtain inequality \eqref{ideal_bounds2}.
Let $z' \in \N(\T_\Omega)$, and let $\eta_{z'}$ and $\E_{z'}$ be as in 
\eqref{discrete_local_problem} and \eqref{comp_local_estimator}. Then,
\begin{align*}
  \E_{z'}^2 = \interleave \eta_{z'} \interleave^2_{\C_{z'}} \! = \int_{\C_{z'}} y^{\alpha} \nabla \eta_{z'} \nabla \eta_{z'} =  
  \int_{\C_{z'}} y^{\alpha} \nabla e_{\T_\Y} \nabla \eta_{z'}
  \leq \interleave  e_{\T_\Y} \! \interleave_{\C_{z'}}   
  \interleave \, \eta_{z'}  \interleave_{\C_{z'}},
\end{align*}
which concludes the proof.
\end{proof}

\begin{remark}[strong efficiency]
The oscillation and constant free lower bound \eqref{dis_lower_bound}
implies a strong concept of efficiency: the relative size of 
$\E_{z'}$ dictates mesh refinement regardless of fine structure of the data $f$, 
and thus works even in the pre-asymptotic regime.
\end{remark}

To obtain an upper bound we must assume the following.

\begin{conjecture}[operator $\calMz$]
\label{conj:hdp}
For every $z' \in\N(\T_\Omega)$ there is a linear operator $\calMz: \W(\C_{z'}) \to \mathcal W(\C_{z'})$ such that, for all $w \in \W(\C_{z'})$, satisfies:
\begin{enumerate}[$\bullet$]
  \item For every cell $K \in \T_\Omega$ such that $K \subset S_{z'}$
  \begin{equation}
  \label{eq:def1zint}
    \int_{K \times \{ 0 \} } \tr \left( w - \calMz w \right) = 0,
  \end{equation}
  \item For every cell $T \subset \C_{z'}$ and every $W \in \V(\T_\Y)$
  \begin{equation}
  \label{eq:def2zint}
    \int_T y^\alpha \nabla\left( w - \calMz w \right)\nabla W  = 0.
  \end{equation}
  \item Stability
    \begin{equation}\label{stab:M}
  \interleave \calMz w \interleave_{\C_{z'}} \lesssim \interleave
  w \interleave_{\C_{z'}},
    \end{equation}
    where the hidden constant is independent of the discretization
    parameters but may depend on $\alpha$.
\end{enumerate} 
\end{conjecture}

We next introduce the so-called \emph{data oscillation}. 
For every $z' \in \N(\T_\Omega)$, 
we define the local data oscillation as
\begin{equation}
\label{eq:defoflocosc}
  \osc_{z'}(f)^2 := d_s h_{z'}^{2s} \| f - f_{z'} \|_{L^2(S_{z'})}^2
\end{equation}
where $f_{z'}|_K \in \R$ is the average of $f$ over $K$, i.e., 
\begin{equation}
 f_{z'}|_K := \fint_{K} f.
\end{equation}
The global data oscillation is defined as
\begin{equation}
\label{eq:defofglobosc}
  \osc_{\T_\Omega}(f)^2 := \sum_{z' \in \N(\T_\Omega)} \osc_{z'}( f)^2 .
\end{equation}
We also define the \emph{total error indicator}
\begin{equation}
\label{total_error}
\tau_{\T_{\Omega}} (V_{\T_{\Y}}, S_{z'}) := \left( \E_{z'}^2 + \osc_{z'}(f)^2 \right)^{1/2} \quad 
\forall z' \in \N(\T_{\Omega}),
\end{equation}
which will be used to mark elements for refinement in the adaptive finite element 
method proposed in Section~\ref{sec:numexp}. Let $\mathscr{K}_{\T_{\Omega}} = \{ S_{z'} : z' \in \N(\T_\Omega)\}$ and,
for any $\mathscr{M} \subset \mathscr{K}_{\T_{\Omega}}$, we set 
\begin{equation}\label{total_est}
\tau_{\T_{\Omega}} (V_{\T_{\Y}}, \mathscr{M} ) := 
 \left( \sum_{S_{z'} \in \mathscr{M} } \tau_{\T_{\Omega}}(V_{\T_{\Y}}, S_{z'})^2 \right)^{1/2}.
\end{equation}

With the aid of the operators $\calMz$ and under the assumption that Conjecture~\ref{conj:hdp} holds, we can bound the
error by the estimator, up to oscillation terms.

\begin{theorem}[global upper bound]
\label{th:upboundcomp}
Let $v \in \HL(\C_\Y,y^\alpha)$ and $V_{\T_\Y} \in \V(\T_\Y)$ 
solve \eqref{alpha_harmonic_extension_weak_T} and \eqref{harmonic_extension_weak}, respectively. 
If $f \in L^2(\Omega)$ and Conjecture~\ref{conj:hdp} holds, then
the total error estimator $\tau_{\T_{\Omega}} (V_{\T_{\Y}},\mathscr{K}_{\T_{\Omega}})$, defined in 
\eqref{total_est} satisfies
\begin{equation}
\label{global_upper_bound}
 \| \nabla (v - V_{\T_\Y}) \|_{L^2(y^\alpha,\C_\Y)} \lesssim 
\tau_{\T_{\Omega}}(V_{\T_\Y},\mathscr{K}_{\T_{\Omega}}).
\end{equation}
\end{theorem}
\begin{proof}
Let $e_{\T_{\Y}}= v - V_{\T_{\Y}}$ denote the error and let
$\psi_{z'} = (w - \Pi_{\T_{\Y}} w) \tilde \varphi_{z'} \in
\W(\C_{z'})$, for any $w \in \HL(y^{\alpha},\C_\Y)$,
where $\Pi_{\T_{\Y}}$ is the quasi-interpolation operator 
introduced in \cite[\S~4]{NOS} and \cite[\S~5]{NOS2}; recall
the estimate $\interleave \psi_{z'} \interleave_{\C_{z'}} \lesssim \interleave w \interleave_{\D_{z'}}$ obtained as part of the proof of Proposition~\ref{pro:ideal_estimator}. Following the proof of Proposition~\ref{pro:ideal_estimator}, we have 
\begin{multline*}
\int_{\C_\Y} y^{\alpha} \nabla e_{\T_\Y} \nabla w 
  =  \sum_{z' \in \N(\T_\Omega)} \int_{\C_{z'}} y^{\alpha} \nabla e_{\T_{\Y}} \nabla \psi_{z'}
 \\
  =  \sum_{z' \in \N(\T_\Omega)} \int_{\C_{z'}} y^{\alpha} \nabla e_{\T_{\Y}} \nabla \calMz \psi_{z'}
   - \sum_{z' \in \N(\T_\Omega)} \int_{\C_{z'}} y^{\alpha} \nabla e_{\T_{\Y}} \nabla (\psi_{z'} - \calMz \psi_{z'}).
\end{multline*}
We now examine each term separately.
First, for every $z' \in \N(\T_\Omega)$ we have $\calMz \psi_{z'} \in
\mathcal{W}(\C_{z'})$, whence
 the definition of the discrete local problem \eqref{discrete_local_problem} yields
 \begin{multline*}
 \sum_{z' \in \N(\T_\Omega)} \int_{\C_{z'}} y^{\alpha} \nabla e_{\T_{\Y}} \nabla \calMz \psi_{z'} 
  = \sum_{z' \in \N(\T_\Omega)} \int_{\C_{z'}} y^{\alpha} \nabla \eta_{z'} \nabla \calMz \psi_{z'} \\
 \leq \left(\sum_{z' \in \N(\T_\Omega)} \E_{z'}^2 \right)^{1/2} \left(\sum_{z' \in \N(\T_\Omega)} 
  \interleave \psi_{z'} \interleave^2_{\C_{z'}} \right)^{1/2}
  \lesssim \E_{\T_\Y} \| \nabla w \|_{L^2(y^\alpha,\C_\Y)},
 \end{multline*}
 where in the last inequality we used the stability assumption
\eqref{stab:M} of the operator $\calMz$, the inequality $\interleave \psi_{z'} \interleave_{\C_{z'}} 
 \lesssim \interleave w \interleave_{\D_{z'}}$, and 
 the finite overlapping property of the stars $\C_{z'}$.
 
 %\item 
Second, for any $z' \in \N(\T_{\Omega})$, we use the conditions \eqref{eq:def1zint} and \eqref{eq:def2zint}
imposed on the operator $\calMz$ to derive
 \[
   \int_{\C_{z'}} y^{\alpha} \nabla e_{\T_{\Y}} \nabla (\psi_{z'} - \calMz \psi_{z'})  = 
   d_s \int_{S_{z'}} (f-f_{z'}) \tr (\psi_{z'} - \calMz \psi_{z'}),
 \]
 where we used that $\nabla V_{\T_\Y}$ is constant on every $T$. Moreover, since $f_{z'}$ is the $L^2(S_{z'})$ projection onto piecewise constants of $f$ we have that, for any $\varrho$ such that $\varrho|_K \in \R$
 \begin{align*}
   \int_{S_{z'}} (f-f_{z'}) \tr (\psi_{z'} - \calMz \psi_{z'}) &=
   \int_{S_{z'}} (f-f_{z'}) \left( \tr (\psi_{z'} - \calMz \psi_{z'}) - \varrho \right) \\
   &\leq 
   \| f - f_{z'} \|_{L^2(S_{z'})} \| \tr (\psi_{z'} - \calMz \psi_{z'}) - \varrho \|_{L^2(S_{z'})}.
%    & \leq
%    h_{z'}^s \| f - f_{z'} \|_{L^2(S_{z'})} \| \tr (\psi_{z'} - \calMz \psi_{z'})\|_{\Hs},
 \end{align*}
 After suitably choosing $\varrho$, and using a standard interpolation-type estimate, we get
\[
 \int_{S_{z'}} (f-f_{z'}) \tr (\psi_{z'} - \calMz \psi_{z'}) \lesssim
 h_{z'}^s \| f - f_{z'} \|_{L^2(S_{z'})} \| \tr (\psi_{z'} - \calMz \psi_{z'})\|_{\Hs}.
\]
Consequently,
 \begin{multline*}
   \sum_{z' \in \N(\T_\Omega)} \int_{\C_{z'}} y^{\alpha} \nabla e_{\T_\Y} \nabla (\psi_{z'} - \calMz \psi_{z'}) \leq \\
   \sum_{z' \in \N(\T_\Omega)} C_{\tr} d_s h_{z'}^s \| f - f_{z'} \|_{L^2( S_{z'})} \interleave  \psi_{z'} - \calMz \psi_{z'}  \interleave_{\C_{z'}}  \lesssim  
   \osc_{\T_\Omega}(f)  \| \nabla w \|_{ L^2(y^{\alpha},\C_\Y) },
 \end{multline*}
where we applied the trace inequality \eqref{Trace_estimate}, the estimate \eqref{eq:estCtr} on $C_{\tr}$, 
the stability assumption \eqref{stab:M} of $\calMz$, the bound
$\interleave \psi_{z'} \interleave_{\C_{z'}} \lesssim \interleave w \interleave_{\D_{z'}}$, and the 
finite overlapping property of the stars $\C_{z'}$.
%\end{enumerate}

Collecting the above estimates, we obtain the asserted bound
  \eqref{global_upper_bound}.
\end{proof}

\begin{remark}[role of oscillation]
\label{remark:osc}
The Definition \ref{def:discrete_spaces} of $\Wcal(\C_{z'})$ is meant
to provide enough degrees of freedom for existence of 
the operator $\calMz$ satisfying \eqref{eq:def1zint}--\eqref{stab:M}.
This leads to a \emph{solution free} oscillation term
\eqref{eq:defoflocosc}. Otherwise, if we were not able to impose
\eqref{eq:def2zint}, then the oscillation \eqref{eq:defoflocosc}
should be supplemented by the term
\[
\osc_{z'}(V_{\T_\Y}) := \| y^\alpha \GRAD V_{\T_\Y} - \bsigma_{z'} \|_{ L^2(y^{-\alpha},\C_{z'}) },
\]
with $\bsigma_{z'}$ being the local average of $y^\alpha \GRAD
V_{\T_\Y}$, for \eqref{global_upper_bound} to be valid. 
This term cannot be guaranteed to be of higher order
due to the presence of the weight $y^{-\alpha}$. In fact, 
computations show that, for $s>\tfrac12$ (or $\alpha<0$),
the magnitude of $\osc_{z'}(V_{\T_\Y})$ 
is of lower order than the actual error estimator $\E_{z'}$ unless
a stronger mesh grading in the extended direction is employed to 
control this term. This can be achieved, for instance, by taking the 
largest grading parameter $\gamma$ in \eqref{graded_mesh} 
for $\pm\alpha$, namely $\gamma>3/(1-|\alpha|)$. Numerical
experiments show no degradation of the convergence rate, expressed in terms of
degrees of freedom, for the ensuing meshes $\T_\Y$ with stronger
anisotropic refinement.
\end{remark}

\begin{remark}[construction of $\calMz$]
The construction of the operator $\calMz$ is an open problem. We design the local space $\Wcal(\C_{z'})$ in order
to provide enough degrees of freedom for the existence of the operator $\calMz$ satisfying \eqref{eq:def1zint}--\eqref{stab:M}.
The numerical experiments of Section \ref{sec:numexp} provide consistent computational evidence that the upper bound \eqref{global_upper_bound} is valid without $\osc_{z'}(V_{\T_\Y})$, and thus indirect evidence of the existence of $\calMz$ with the requisite properties \eqref{eq:def1zint}--\eqref{stab:M}.
\end{remark}

\section{Numerical experiments}
\label{sec:numexp}
Here we explore computationally the performance and limitations of the a posteriori error 
estimator introduced in \S\ref{subsec:computable} with a series of 
test cases.
To do so, we start by formulating an adaptive finite element method (AFEM) based on iterations of the loop
\begin{equation}
 \textsf{\textup{SOLVE}} \rightarrow \textsf{\textup{ESTIMATE}} \rightarrow \textsf{\textup{MARK}} \rightarrow \textsf{\textup{REFINE}}.
\label{afem}
\end{equation}

\subsection{Design of AFEM}
The modules in \eqref{afem} are as follows:
\begin{enumerate}[$\bullet$]
 \item \textsf{\textup{SOLVE}}: Given a mesh $\T_{\Y}$ we compute the Galerkin solution of \eqref{alpha_harmonic_extension_weak_T}:
\[
 V_{\T_{\Y}} = \textsf{\textup{SOLVE}}(\T_{\Y}).
\]

\item \textsf{\textup{ESTIMATE}}: Given $V_{\T_\Y}$ we calculate 
the local error indicators \eqref{comp_local_estimator} and the local oscillations \eqref{eq:defoflocosc}
to construct  the total error indicator of \eqref{total_error}:
\[
 \left\{ \tau_{\T_{\Omega}} (V_{\T_{\Y}}, S_{z'}) \right\}_{ S_{z'} \in \mathscr{K}_{\T_\Omega}} = \textsf{\textup{ESTIMATE}}(V_{\T_{\Y}}, \T_{\Y}).
\]

\item \textsf{\textup{MARK}}: Using the so-called D\"{o}rfler marking \cite{MR1393904} (bulk chasing strategy) 
with parameter $0< \theta \leq 1$, we select a set 
\[
  \mathscr{M} = \textsf{\textup{MARK}}( \left\{ \tau_{\T_{\Omega}} (V_{\T_{\Y}}, S_{z'}) \right\}_{ S_{z'} \in \mathscr{K}_{\Omega}}, V_{\T_{\Y}} ) 
  \subset \mathscr{K}_{\T_{\Omega}}
\]
of minimal cardinality that satisfies
\[
  \tau_{\T_{\Omega}}( V_{\T_{\Y}},\mathscr{M})  \geq \theta \tau_{\T_{\Omega}}(V_{\T_{\Y}},  \mathscr{K}_{\T_{\Omega}}). 
\]

\item \textsf{\textup{REFINE}}: We generate a new mesh $\T_\Omega'$ by bisecting all the elements 
$K \in \T_{\Omega}$ contained in $\mathscr{M}$ based on the newest-vertex bisection method; see \cite{NV,NSV:09}. 
We choose the truncation parameter as $\Y = 1 + \tfrac{1}{3}\log(\# \T_{\Omega}')$ to balance the approximation and truncation errors; see \cite[Remark 5.5]{NOS}.
We construct the mesh $\mathcal{I}_\Y'$ by the rule \eqref{graded_mesh}, 
with a number of degrees of freedom $M$ 
sufficiently large so that condition \eqref{assumption} holds.
This is attained by first creating a partition 
$\mathcal{I}_\Y'$ with $M \approx (\# \T_\Omega')^{1/n}$
and checking \eqref{assumption}. If this condition is violated, we increase 
the number of points until we get the desired result. The new mesh
\[
  \T_\Y' = \textsf{\textup{REFINE}}(\mathscr{M}),
\]
is obtained as the tensor product of $\T_\Omega'$ and $\mathcal{I}_\Y'$.
\end{enumerate}

\subsection{Implementation}
The AFEM \eqref{afem} is implemented within
the MATLAB$^\copyright$ software library {\it{i}}FEM~\cite{Chen.L2008c}.
The stiffness matrix of the discrete system \eqref{harmonic_extension_weak} is assembled exactly, and
the forcing boundary term is computed by a quadrature formula which is exact for polynomials of degree $4$. 
The resulting linear system is solved by using the multigrid method with line smoother introduced
and analyzed in~\cite{CNOS}.

To compute the solution $\eta_{z'}$ to the discrete local problem
\eqref{discrete_local_problem}, we loop around each node $z' \in \N(\T_{\Omega})$, 
collect data about the cylindrical star $\C_{z'}$ and assemble the small linear
system \eqref{discrete_local_problem} which is solved by the built-in 
\emph{direct solver} of MATLAB$^\copyright$. 
All integrals involving only the weight and discrete functions
are computed exactly, whereas those also involving data functions are computed
element-wise by a quadrature formula which is exact for polynomials of degree 7. 

For convenience, in the \textsf{\textup{MARK}} step
we change the estimator from star-wise to element-wise 
as follows: We first scale the nodal-wise estimator as $\E_{z'}^2 / (\# S_{z'} )$ and then,
for each element $K\in \T_{\Omega}$, we compute 
\[
\E_{K}^2 := \sum_{z'\in K}\E_{z'}^2. 
\]
The scaling is introduced so that 
$\sum _{K \in \T_{\Omega} }\E_{K}^2 = \sum _{z' \in \N(\T_\Omega) } \E_{z'}^2$. 
The cell-wise data oscillation is now defined as
\[
\osc_{K}(f)^2 := d_sh_{K}^{2s}\|f - \bar f_K\|^2_{L^2(K)}, 
\]
where $\bar f_K$ denotes the average of $f$ over the element $K$. This quantity is 
computed using a quadrature formula which is exact for polynomials of degree 7. 

Unless specifically mentioned, all computations are done without 
explicitly enforcing the mesh restriction \eqref{assumption}, which shows that this is nothing
but an artifact in our proofs; see Remark \ref{relax_mesh}. How to remove this assumption is currently under investigation. 
Nevertheless, computations (not shown here for brevity) show that optimality is still retained 
if one imposes \eqref{assumption}. 

For the cases where the exact solution to problem \eqref{harmonic_extension_weak} is available,
the error is computed by using the identity
\begin{equation*}
\| \nabla(v-V_{\T_{\Y}})\|_{L^2(y^{\alpha},\C_{\Y})}^2 = d_s\int_{\Omega}f \tr (v-V_{\T_{\Y}}),
\end{equation*}
which follows from Galerkin orthogonality and integration by parts. 
Thus, we avoid evaluating the singular/degenerate weight $y^{\alpha}$
and reduce the computational cost. The right hand side of the equation above is  
computed by a quadrature formula which is exact for polynomials of degree 7.
On the other hand, if the exact solution is not available, we introduce the energy
\[
E(w) = \frac{1}{2}\int_{\mathcal \C_{\Y}} y^{\alpha}| \nabla w |^2  - d_s \langle f, \tr w \rangle_{\Hsd \times \Hs}.
\]
where $w \in \HL(y^{\alpha},\C_{\Y})$ and $f \in \Hsd$. Consequently, Galerkin orthogonality implies
$$
E(V_{\T_{\Y}}) - E(v) = \frac{1}{2}\| \nabla(v-V_{\T_{\Y}})\|_{L^2(y^{\alpha},\C_{\Y})}^2,
$$
where $v \in  \HL(y^{\alpha},\C_{\Y})$ and $V_{\T_{\Y}} \in \V(\T_{\Y})$ denote the solution
to problems \eqref{alpha_harmonic_extension_weak_T} and \eqref{harmonic_extension_weak}, respectively.
We remark that for a discrete function $W$ the energy 
$E(W)$ can be computed simply using a matrix-vector product.
Then, an approximation of the error in the energy norm can be obtained by computing 
\[
\left( 2( E(V_{\T_{\Y}}) - E(V_{\T_{\Y}^*}) ) \right)^{\tfrac{1}{2}},
\]
where $V_{\T_{\Y}^*}$ is the Galerkin approximation to $v$ on a very fine grid 
$\T_{\Y}^{*}$, which is obtained by a uniform refinement of the last adaptive mesh.

We also compute the \emph{effectivity index} of our a posterior error estimator defined as the ratio of the estimator and the true error, \ie 
\[
\frac{\tau_{\T_{\Omega}}(V_{\T_\Y},\mathscr{K}_{\T_{\Omega}})}{\| \nabla(v-V_{\T_{\Y}})\|_{L^2(y^{\alpha},\C_{\Y})}},
\]
as well as the \emph{aspect ratio} of elements $T$
\[
 \frac{h_K}{h_{I}} \qquad \forall T = K \times I \in \T_{\Y}.
\]
 
\subsection{Smooth and compatible data}
\label{sub:smoothcompatible}

The purpose of this numerical example is to show how the error estimator
\eqref{comp_global_estimator}--\eqref{comp_local_estimator} based on the solution
of local problems on anisotropic cylindrical stars allows us to recover optimality of our AFEM. 
We recall that adaptive isotropic refinement cannot be optimal; see \S \ref{subsec:isotropic}.
We consider $\Omega = (0,1)^2$ and $f(x^1,x^2)= \sin (2\pi x^1) \sin (2 \pi x^2)$. 
Then, the exact solution to \eqref{fl=f_bdddom} is given by $u(x^1,x^2) = 8^{-s}\pi^{2s}\sin (2\pi x^1) \sin (2 \pi x^2)$.

The asymptotic relation $\|\nabla( \ue - V_{\T_{\Y}})\|_{L^2(y^{\alpha},\C_{\Y})} 
\approx \#(\T_\Y)^{-1/3}$
is shown in Figure~\ref{fig:smooth_function} which illustrates the quasi-optimal 
decay rate of our AFEM driven by the error estimator 
\eqref{comp_global_estimator}--\eqref{comp_local_estimator} for all choices of the parameter $s$ considered. 
These examples show that anisotropy in the extended dimension is essential to recover optimality of our AFEM. 

\begin{figure}[ht]
  \begin{center}
  \includegraphics[width=0.49\textwidth]{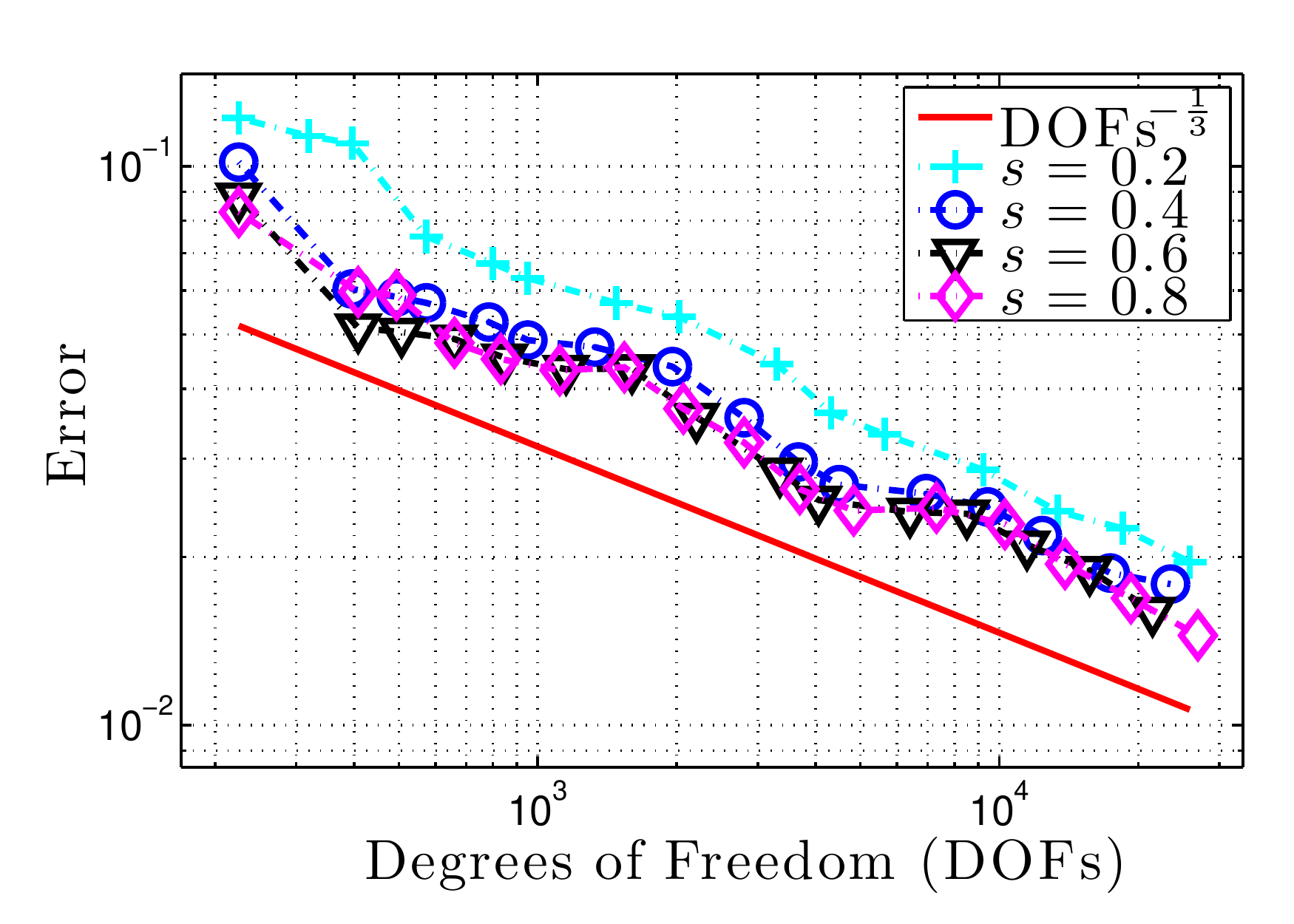}
  \hfil
  \includegraphics[width=0.49\textwidth]{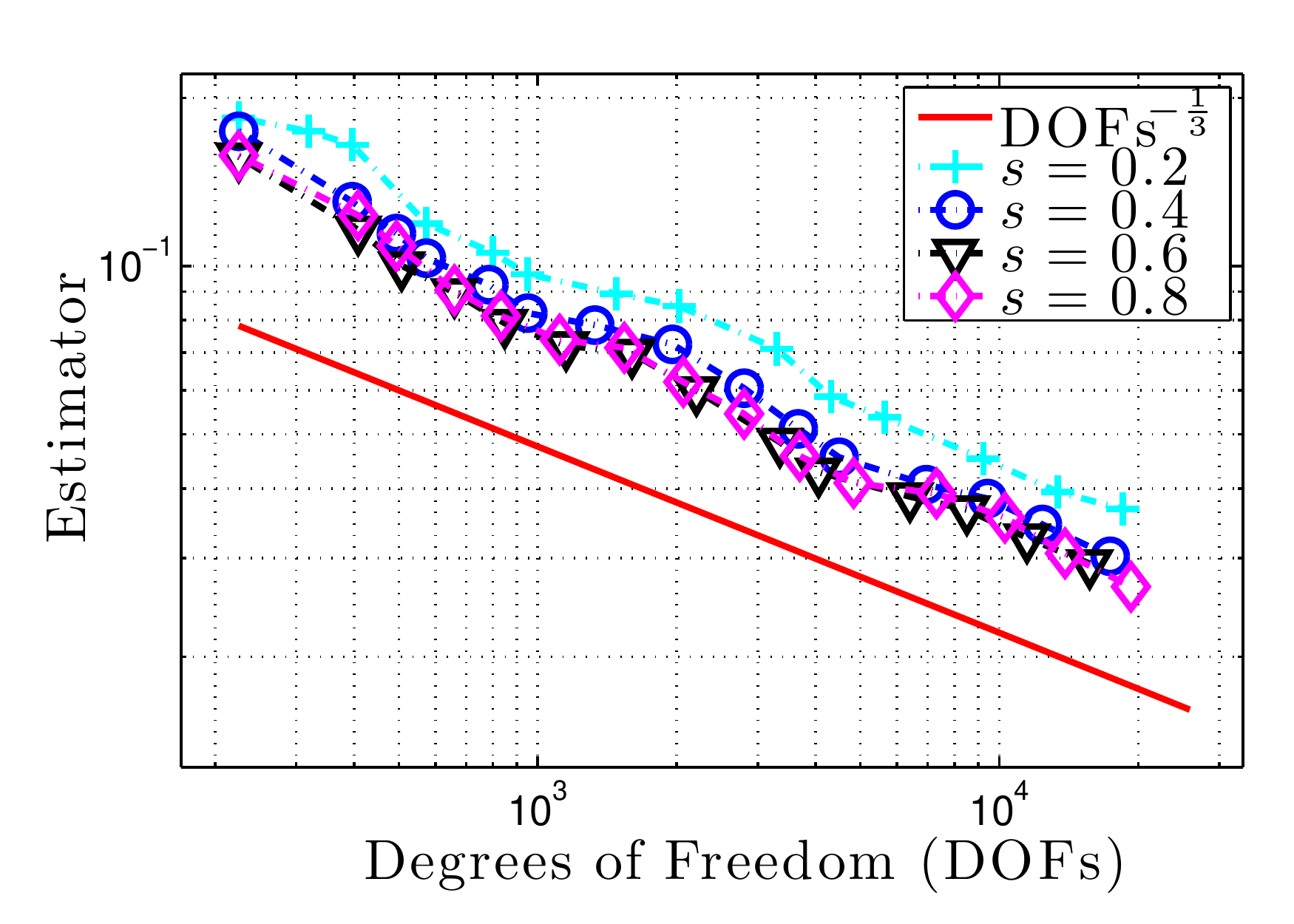}
  \end{center}
  \caption{Computational rate of convergence for our anisotropic AFEM on the
smooth and compatible right hand side of \S~\ref{sub:smoothcompatible} for $n=2$ and $s=0.2$, $0.4$, $0.6$ and $s=0.8$. The left panel 
shows the decrease of the error with respect to the number of degrees of freedom, whereas the right one 
that for the total error estimator. In all cases we recover the optimal rate $\# (\T_{\Y})^{-1/3}$. 
The aspect ratios, averaged over $x'$, of the cells on the bottom layer $[0,y_1]$  in the finest mesh are: $1.65\times 10^{11}$, $6.30\times 10^{4}$, $396$ and $36.2$, respectively.
The average effectivity indices are $1.47, 1.61, 1.61, 1.62$, respectively.
}
\label{fig:smooth_function}
\end{figure}

\subsection{Smooth but incompatible data}
\label{sub:smoothincompatible}

The numerical example presented in \cite[\S6.3]{NOS} shows that the results of 
Theorem~\ref{TH:fl_error_estimates} are sharp in the sense that the regularity 
$f\in\Ws$ is indeed necessary to obtain an optimal rate of convergence with a 
quasiuniform mesh in the $x'$-direction. The heuristic explanation for this is 
that a certain compatibility between the data and the boundary condition is necessary. 
The results of \cite[\S6.3]{NOS} also show that, in some simple cases, one can guess 
the nature of the singularity that is introduced by this incompatibility and a priori 
design a mesh that captures it, thus recovering the optimal rate of convergence. 
Evidently, this is not possible in all cases and here we show that the a 
posteriori error estimator \eqref{comp_global_estimator}--\eqref{comp_local_estimator} 
automatically produces a sequence of meshes that yield the optimal rate of convergence.

We consider $\Omega = (0,1)^2$ and $f=1$. Then, for $s < \tfrac12$, we have that $f\notin \Ws$ 
whence we cannot invoke the results of Theorem~\ref{TH:fl_error_estimates}. Nevertheless, as 
the results of Figure~\ref{fig:incompatible} show, we recover the optimal rate of convergence.

\begin{figure}[ht]
  \begin{center}
  \includegraphics[width=0.49\textwidth]{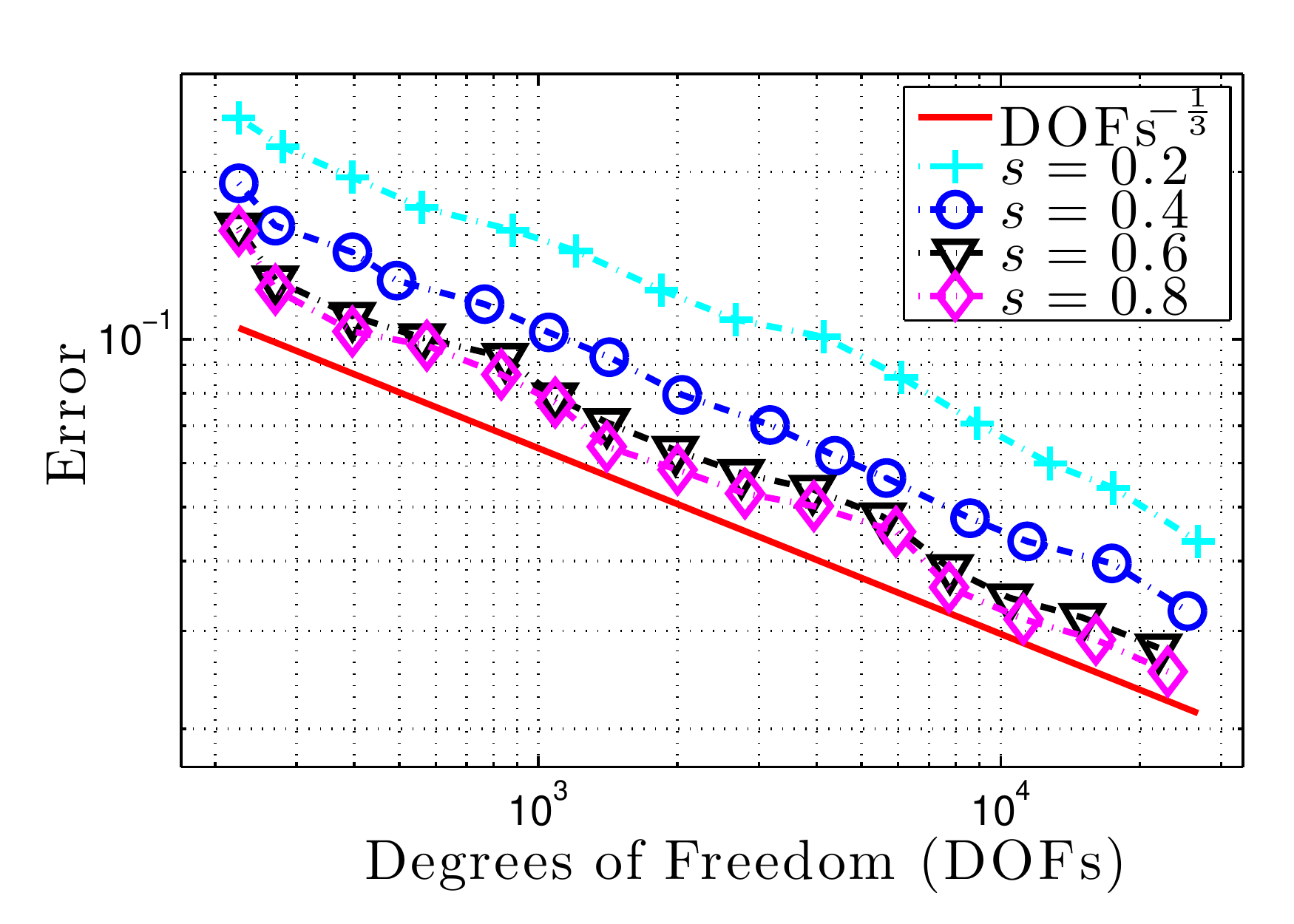}
  \hfil
  \includegraphics[width=0.49\textwidth]{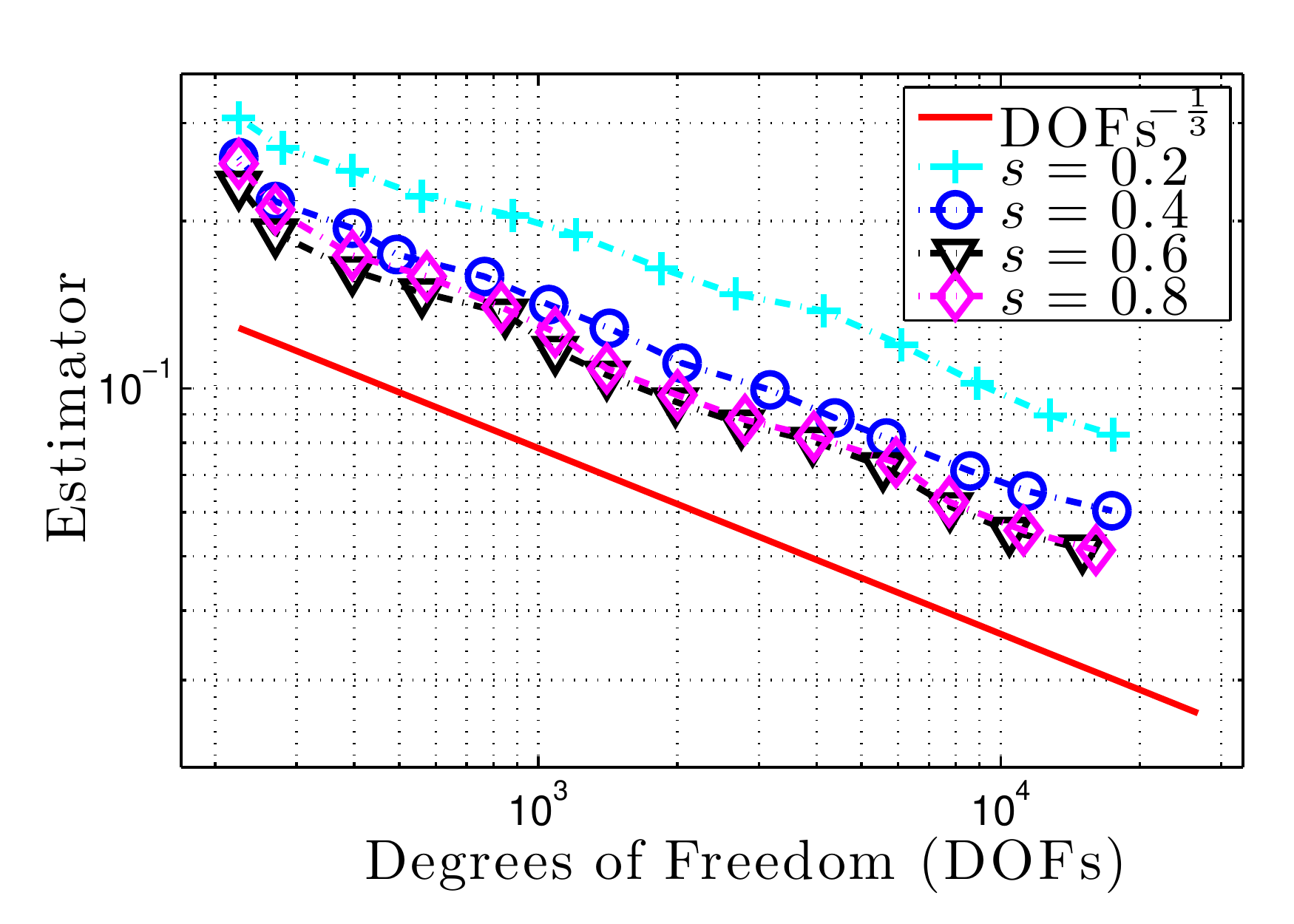}
  \end{center}
  \caption{Computational rate of convergence for our anisotropic AFEM on the
smooth but incompatible right hand side of \S~\ref{sub:smoothincompatible} for $n=2$ and $s=0.2$, $0.4$, $0.6$ and $s=0.8$. 
The left panel shows the decrease of the error with respect to the number of degrees of freedom, 
whereas the right one that for the total error indicator. 
In all cases we recover the optimal rate $\# (\T_{\Y})^{-1/3}$. Notice that, for $s<\tfrac12$, 
the right hand side $f=1\notin \Ws$ and so a quasiuniform mesh in $\Omega$ does not deliver the optimal 
rate of convergence \cite[\S6.3]{NOS}.
The aspect ratios, averaged over $x'$, of the cells on the bottom layer $[0,y_1]$  in the finest mesh are: $2.09\times 10^{11}$, 
$6.03\times 10^{4}$, $387$ and $33.4$, respectively.
The average effectivity indices are $1.34$, $1.41$, $1.51$, $1.67$, respectively.}
\label{fig:incompatible}
\end{figure}

\subsection{L-shaped domain with compatible data}
\label{sub:Lshapedcompatible}

The result of Theorem~\ref{TH:fl_error_estimates} relies on the assumption that the Laplace operator 
$-\Delta_{x'}$ on $\Omega$ supplemented with Dirichlet boundary conditions possesses optimal smoothing properties, 
\ie \eqref{reg_Omega}. As it is well known \cite{Grisvard}, \eqref{reg_Omega} holds when 
$\Omega$ is a convex polyhedron. Let us then consider the case when this condition is not met.

We consider the classical L-shaped domain $\Omega = (-1,1)^2\setminus(0,1)\times(-1,0)$ 
and $f(x^1,x^2)=\sin(2\pi x^1)\sin(\pi x^2)$, which is a smooth and compatible right hand side, i.e., $f = 0$ on $\partial \Omega$, for 
all values of $s\in(0,1)$. The results of Figure~\ref{fig:Lshapedcompatible} show the 
experimental rates of convergence and confirm that our AFEM is able to 
capture the singularity that the reentrant corner in the domain introduces and recover 
the optimal rate of convergence. As the exact solution of this problem is not known, 
we compute the rate of convergence by comparing the computed solution with one obtained 
on a very fine mesh.

\begin{figure}[ht]
  \begin{center}
  \includegraphics[width=0.49\textwidth]{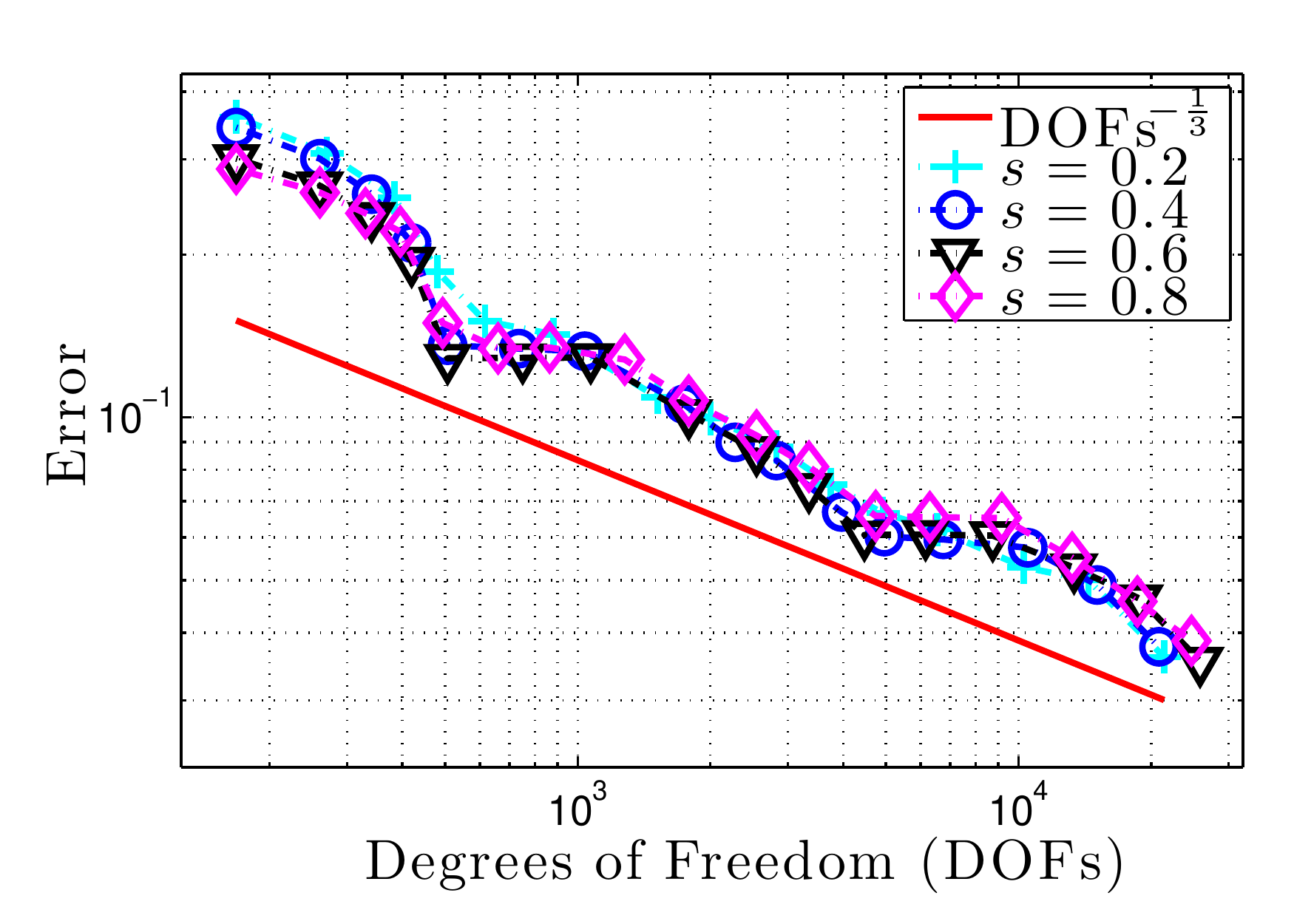}
  \hfil
  \includegraphics[width=0.49\textwidth]{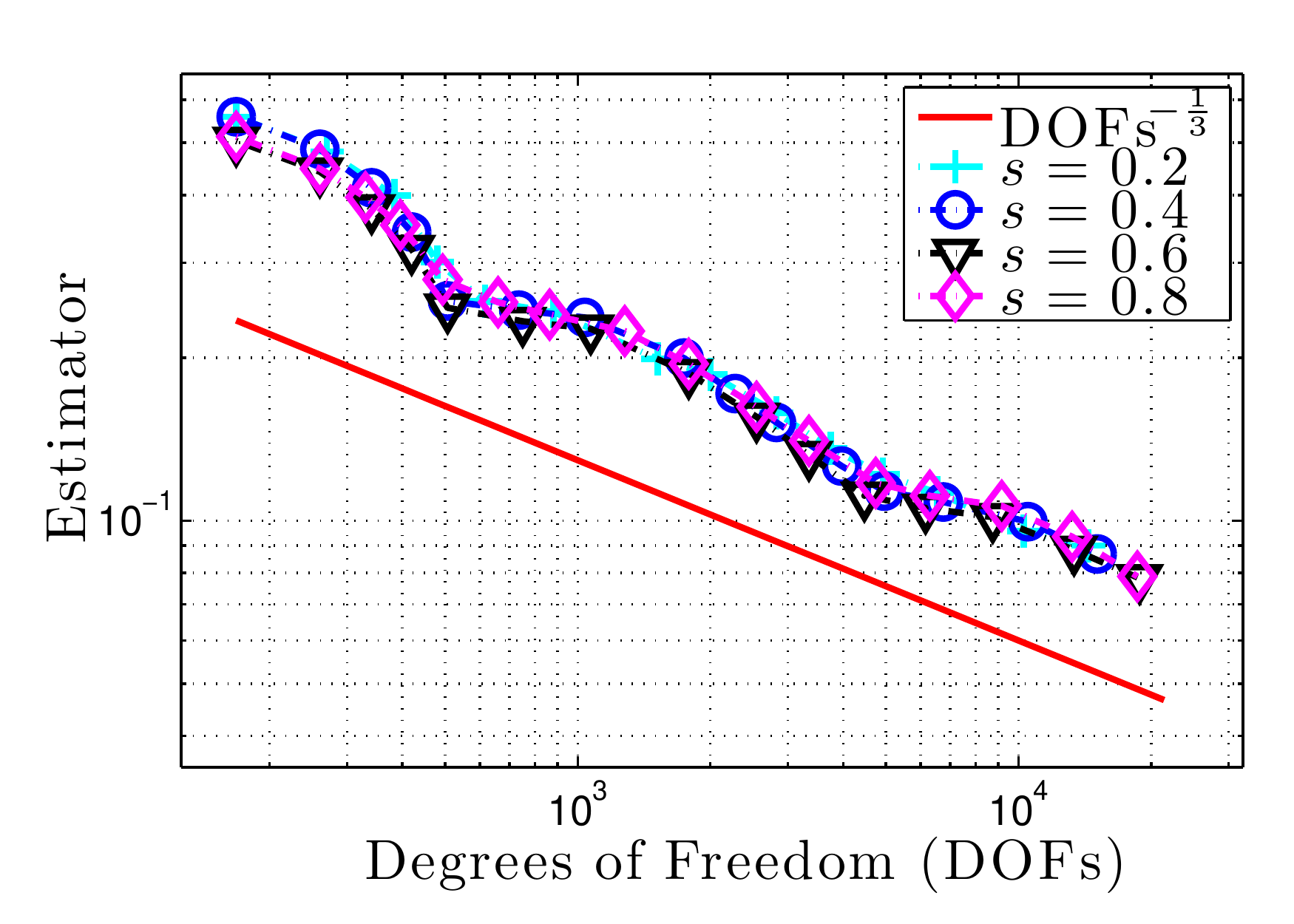}
  \end{center}
  \caption{Computational rate of convergence for our anisotropic AFEM on the smooth and 
compatible right hand side over an L-shaped domain of \S~\ref{sub:Lshapedcompatible} for $n=2$ and $s=0.2$, $0.4$, $0.6$ and $s=0.8$.
The left panel shows the decrease of the error with respect to the number of degrees of freedom, whereas the right 
one that for the total error indicator. 
In all cases we recover the optimal rate $\# (\T_{\Y})^{-1/3}$. As the exact solution of this problem is not known, 
we compute the rate of convergence by comparing the computed solution with one obtained on a very fine mesh.
The aspect ratios, averaged over $x'$, of the cells on the bottom layer $[0,y_1]$ in the finest mesh are: $2.56\times 10^{11}$, 
$1.20\times 10^{5}$, $667$ and $56.7$, respectively.
The average effectivity indices are $1.73$, $1.77$, $1.73$, $1.74$, respectively.}
\label{fig:Lshapedcompatible}
\end{figure}

\subsection{L-shaped domain with incompatible data}
\label{sub:Lshapedincompatible}

Let us combine the singularity introduced by the 
data incompatibility of \S~\ref{sub:smoothincompatible}
and the reentrant corner explored in 
\S~\ref{sub:Lshapedcompatible}. 
To do so, we again consider the L-shaped domain $\Omega = (-1,1)^2\setminus(0,1)\times(-1,0)$ 
and $f=1$. As the results of Figure~\ref{fig:Lshapedincompatible} show, we again recover the 
optimal rate of convergence for all possible cases of $s$. As the exact solution of this problem is not known, 
we compute the rate of convergence by comparing the computed solution with one obtained 
on a very fine mesh.

\begin{figure}[ht]
  \begin{center}
  \includegraphics[width=0.49\textwidth]{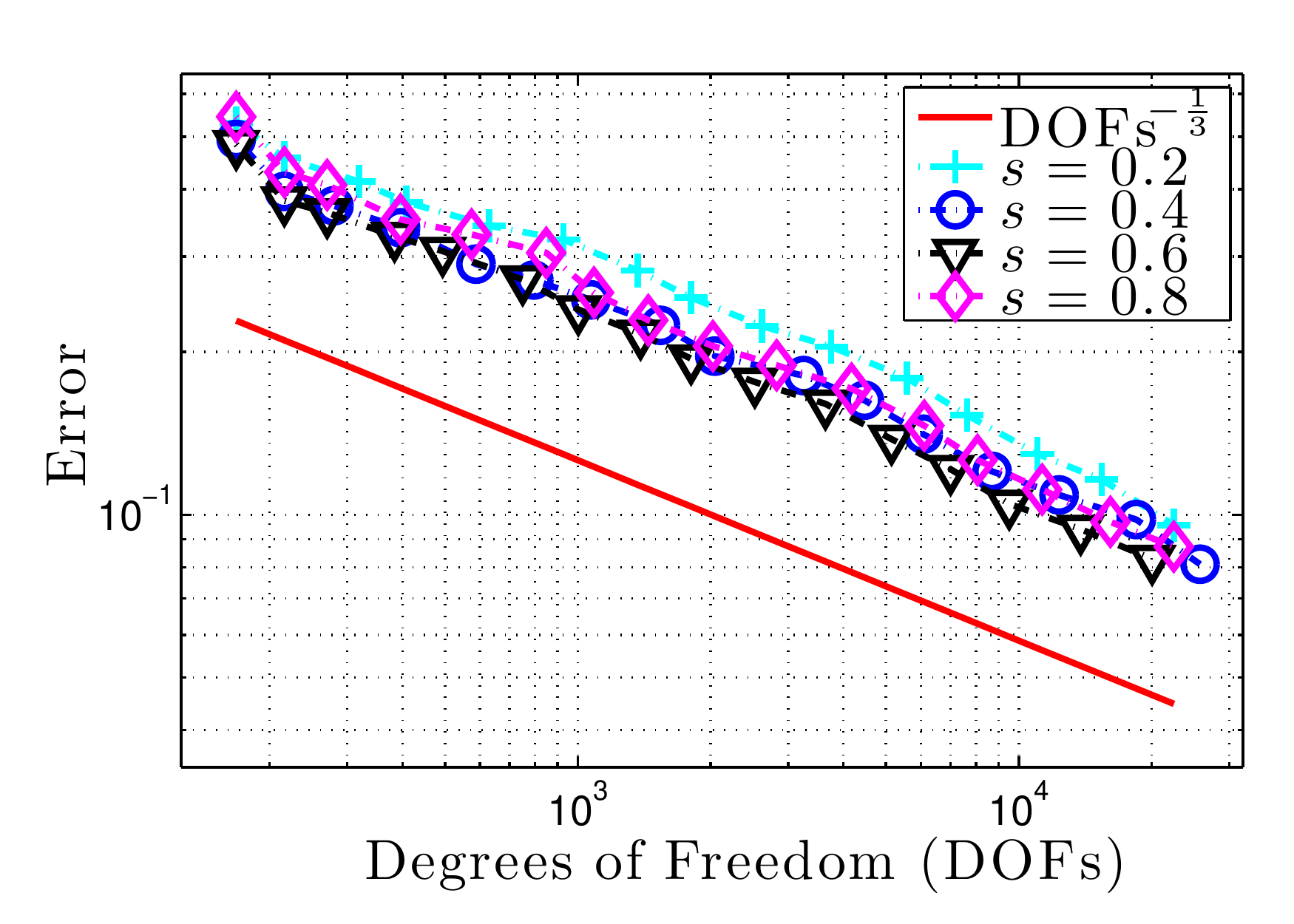}
  \hfil
  \includegraphics[width=0.49\textwidth]{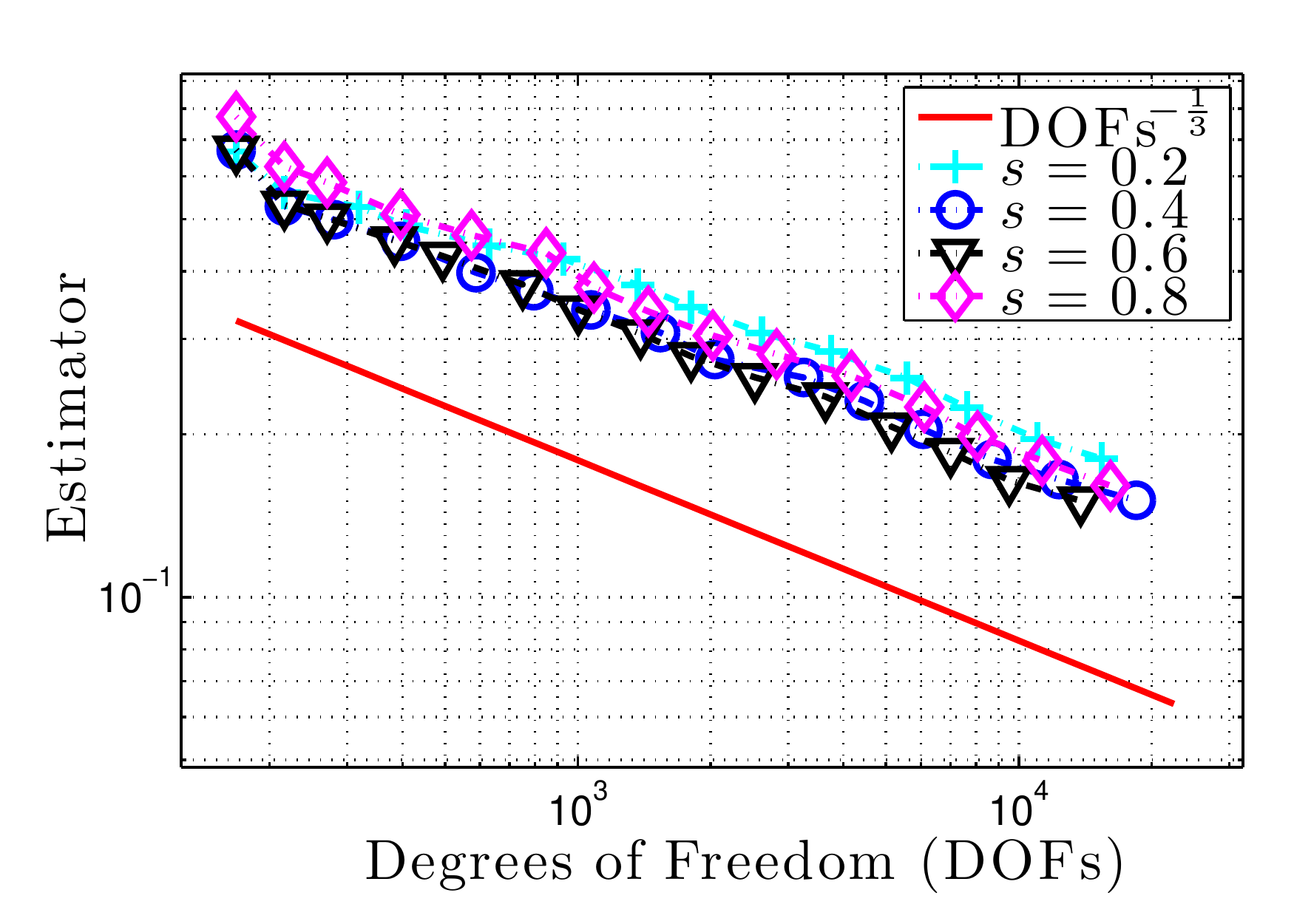}
  \end{center}
  \caption{Computational rate of convergence for our anisotropic AFEM on the smooth but 
incompatible right hand side over an L-shaped domain of \S~\ref{sub:Lshapedincompatible} for $n=2$ and $s=0.2$, $0.4$, $0.6$ and $s=0.8$. 
The left panel shows the decrease of the error with respect to the number of degrees of freedom, 
whereas the right one that for the total error indicator. 
In all cases we recover the optimal rate $\# (\T_{\Y})^{-1/3}$. As the exact solution of this problem is not 
known, we compute the rate of convergence by comparing the computed solution with one obtained 
on a very fine mesh.
The aspect ratios, averaged over $x'$, of the cells on the bottom layer $[0,y_1]$ in the finest mesh are: $2.24\times 10^{11}$, 
$1.02\times 10^{5}$, $632$ and $55.6$, respectively.
The average effectivity indices are $1.36$, $1.40$, $1.44$, $1.49$, respectively.}
\label{fig:Lshapedincompatible}
\end{figure}

We display some meshes in Figure~\ref{fig:meshLshapedincompatible}. As expected, when $s<\tfrac12$ the data 
$f=1$ is incompatible with the equation and this causes a boundary layer on the solution. To capture it, 
the AFEM refines near the  boundary. In contrast, when $s>\tfrac12$ the refinement is only 
near the reentrant corner, since there is no boundary layer anymore.

\begin{figure}[ht]
  \begin{center}
  \includegraphics[width=0.25\textwidth]{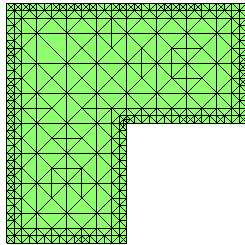}
  \hfil
  \includegraphics[width=0.25\textwidth]{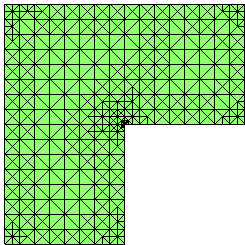}
  \end{center}
  \caption{Adaptively graded mesh for an L-shaped domain with incompatible data (see \S~\ref{sub:Lshapedincompatible}): $s=0.2$ (left) and 
  $s = 0.8$ (right). As expected, when $s<\tfrac12$ the data $f=1$ is incompatible with the equation 
  and this causes a boundary layer on the solution. To capture it,
  our AFEM refines near the 
  boundary. In contrast, when $s>\tfrac12$ the refinement is only near the reentrant angle, since there is no boundary layer anymore.}
\label{fig:meshLshapedincompatible}
\end{figure}

\subsection{Discontinuous coefficients}
\label{sub:disccoeff}

In this example we compute the fractional powers of a general elliptic operator with 
piecewise discontinuous coefficients. We invoke the formulas derived by
Kellogg \cite{Kellog} (see also \cite[\S5.4]{NV}) to construct an exact solution to the particular case:
% To do so, we consider the classical example due to Kellog \cite{Kellog} 
% (see also \cite[\S5.4]{NV}). 
$\Omega = (-1,1)^2$, $f(x^1,x^2)=( (x^1)^2 - 1)( (x^2)^2-1)$  and
\[
  \calL w = -\DIV_{x'}(A \nabla_{x'}w),
\]
with
\[
  A = \begin{dcases}
        \varrho\mathbb{I}, & x^1x^2 >0, \\
        \mathbb{I}, & x^1x^2 \leq 0,
      \end{dcases}
\]
where $\mathbb{I}$ denotes the identity tensor and $\varrho=161.4476387975881$. 
This problem is well known as a pathological case, where the solution is \emph{barely in} $H^1_0(\Omega)$.

On the other hand, as the results of 
\cite[\S7]{NOS} show, the problem: find $u$ such that
\[
  \calL^s u = f,\  \text{in } \Omega \qquad u|_{\partial\Omega} = 0
\]
also admits an extension, which can be discretized with the techniques and ideas presented in 
\cite[\S7]{NOS}. We can also write an a posteriori error estimator based on cylindrical 
stars and design an adaptive loop. 
Figure~\ref{fig:Kellog} demonstrates that the associated decay rate is optimal: 
$(\T_{\Y})^{-1/3}$ for all the considered cases.

\begin{figure}[ht]
  \begin{center}
  \includegraphics[width=0.49\textwidth]{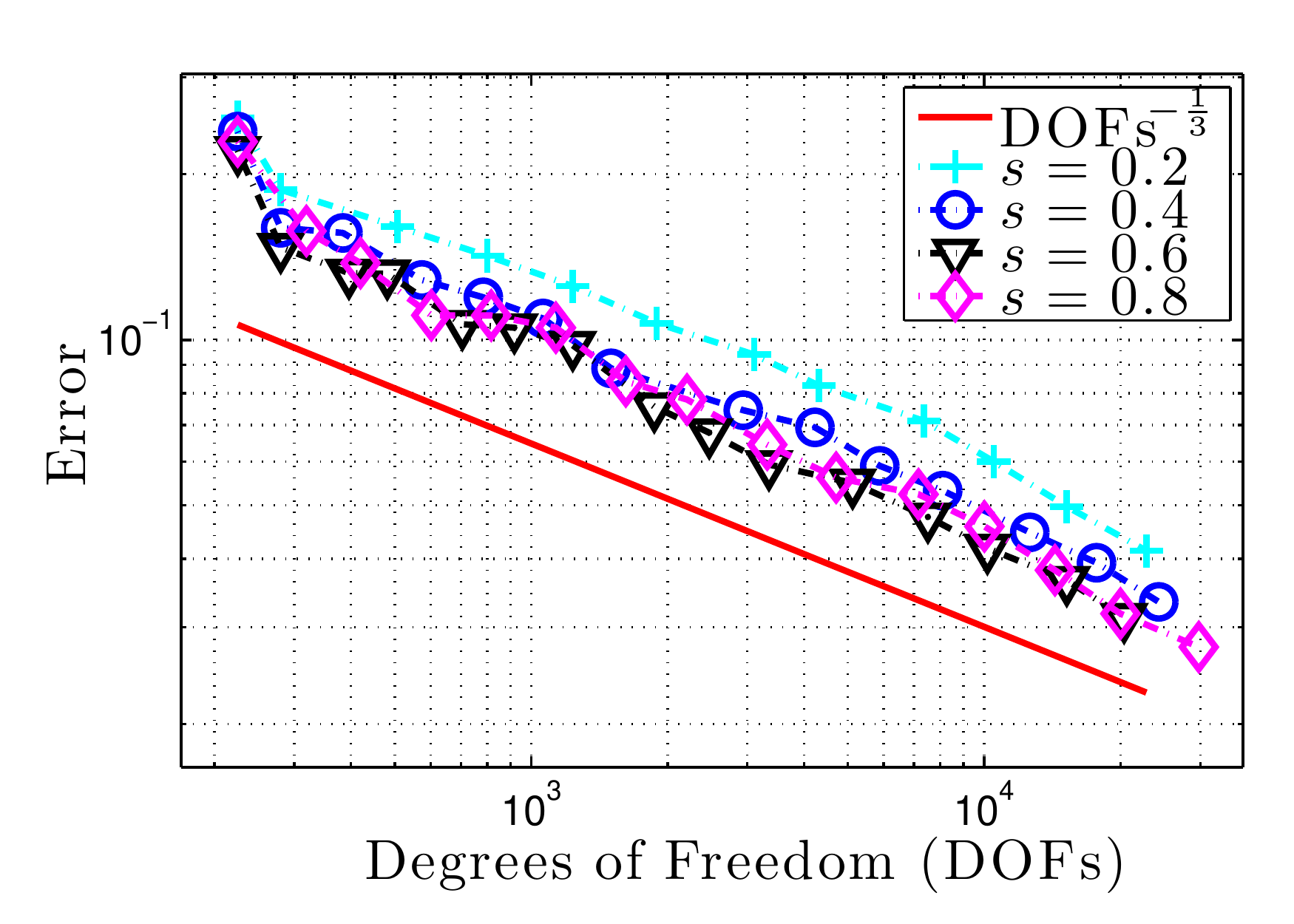}
  \hfil
   \includegraphics[width=0.49\textwidth]{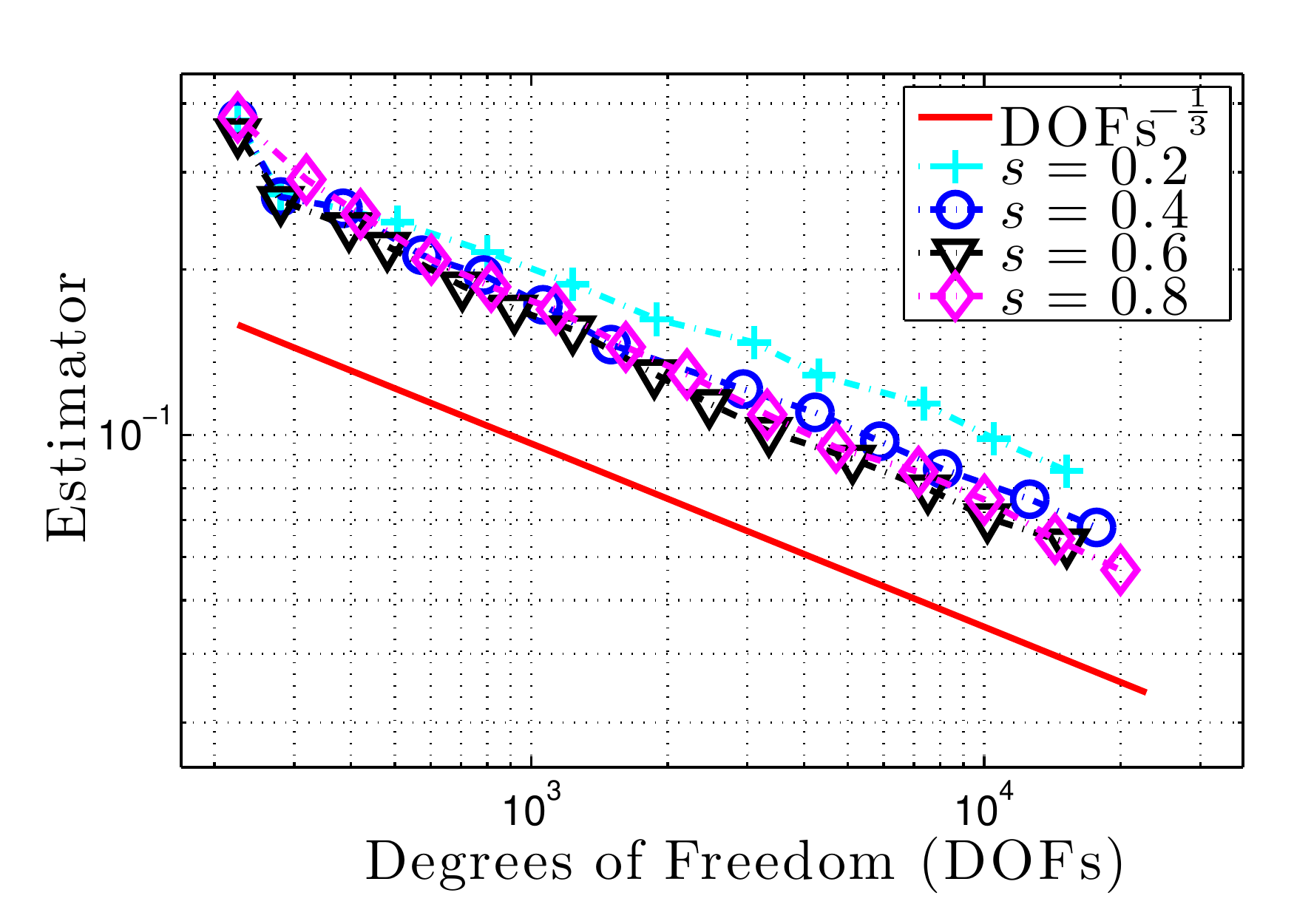}
  \end{center}
  \caption{Computational rate of convergence for our anisotropic AFEM on the problem with 
discontinuous coefficients described in \S~\ref{sub:disccoeff} for $n=2$ and $s=0.2$, $0.4$, $0.6$ and $s=0.8$. 
The left panel shows the decrease of the error with respect to the number of degrees of freedom, whereas the 
right one that for the total error indicator.
In all cases we recover the optimal rate $\# (\T_{\Y})^{-1/3}$.
As the exact solution of this problem is not known, we compute the rate of convergence by 
comparing the computed solution with one obtained on a very fine mesh. 
The aspect ratios, averaged over $x'$, of the cells on the bottom layer $[0,y_1]$ in the finest mesh are: $3.74\times 10^{11}$, 
$1.21\times 10^{5}$, $719$ and $69.1$, respectively.
The average effectivity indices are $1.55$, $1.64$, $1.69$, $1.71$, respectively.}
\label{fig:Kellog}
\end{figure}

\subsection{The role of oscillation} 
\label{sub:osc}
Let us explore the role of oscillation as discussed in Remark~\ref{remark:osc}. 
We consider the same example as in \S~\ref{sub:Lshapedincompatible}, 
but we set in Definition~\ref{def:discrete_spaces} $\mathcal{P}_2(K) = \mathbb{P}_2(K)$. In this case the discrete local space $\mathcal{W}(\C_{z'})$ does not have enough degrees of freedom to impose \eqref{eq:def2zint}. Consequently, in order to have \eqref{global_upper_bound} the oscillation \eqref{eq:defoflocosc} must be supplemented by the term
\begin{equation}
\label{osc_V}
\osc_{z'}(V_{\T_\Y}) = \| y^\alpha \GRAD V_{\T_\Y} - \bsigma_{z'} \|_{L^2(\C_{z'},y^{-\alpha})},
\end{equation}
where $\bsigma_{z'}$ is the local average of $y^\alpha \GRAD V_{\T_\Y}$.

\begin{figure}[ht]
  \begin{center}
   \includegraphics[width=0.49\textwidth]{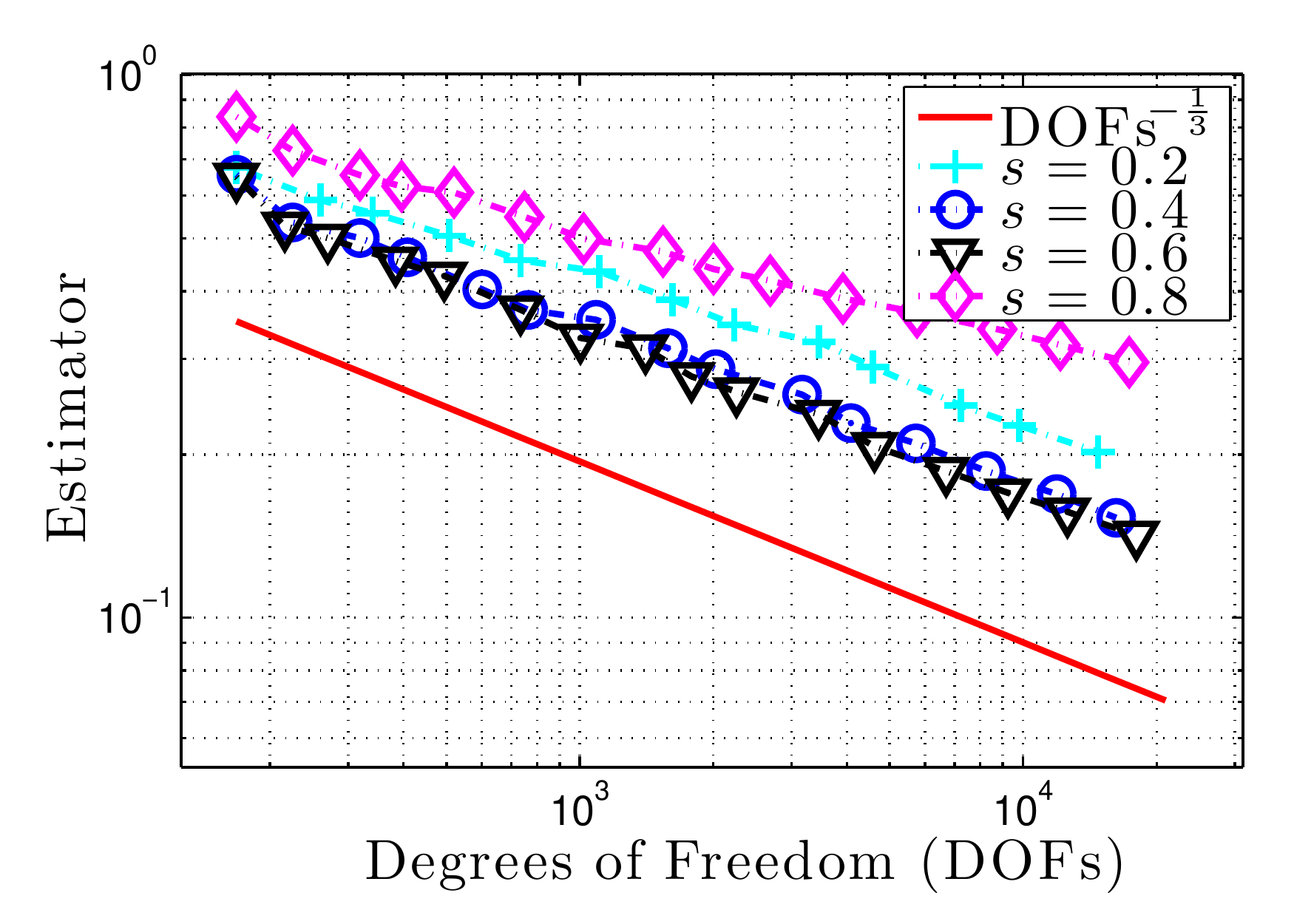}
   \hfil
   \includegraphics[width=0.49\textwidth]{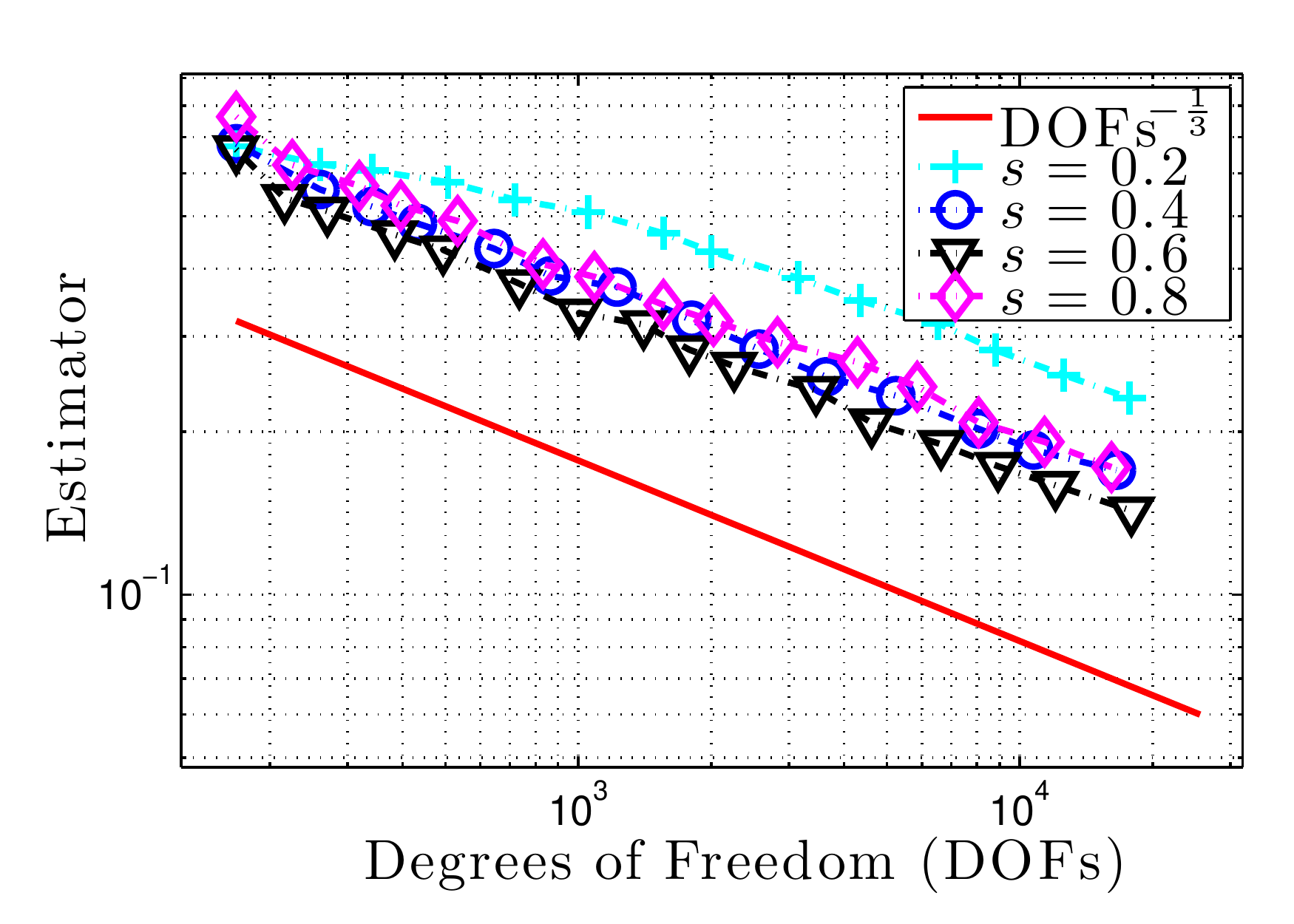}
  \end{center}
\caption{Experimental rate of convergence for the example of
  \S~\ref{sub:Lshapedincompatible} but with $\mathcal{P}_2(K) =
  \mathbb{P}_2(K)$ in Definition~\ref{def:discrete_spaces}. In this
  case, the oscillation \eqref{eq:defoflocosc} must be supplemented by
  \eqref{osc_V} to attain an upper bound. The expression
  \eqref{osc_V} possesses a singularity and weight which cannot be
  captured with the mesh grading necessary for optimal approximation
  orders; this yields suboptimal decay rates (left figure).
  However, a graded mesh with $\gamma > 3/(1-|\alpha|)$ in
  \eqref{graded_mesh} gives optimal decay rates (right figure).}
\label{fig:osc}
\end{figure}

Figure~\ref{fig:osc} shows the experimental rates of convergence
obtained by our AFEM for the total error where the oscillation term
\eqref{eq:defoflocosc} is supplemented with \eqref{osc_V}. As we can
see, especially for $s=0.8$, the results are not optimal. To remedy
this we notice that a graded mesh is able to capture the singular
behavior of \eqref{osc_V}. Indeed, the underlying weight in this
expression is $y^{-\alpha}$, so that a grading parameter in
\eqref{graded_mesh} of $\gamma > 3/(1-|\alpha|)$ would yield an
optimal decay for \eqref{osc_V}. For this reason, setting
$\gamma>3/(1-|\alpha|)$, we are able to obtain an optimal decay rates
for both the error and the oscillation \eqref{osc_V} and
all values of $s$. The results shown in Figure~\ref{fig:osc} confirm this.
%============ References ==================
\bibliographystyle{plain}
\bibliography{biblio}

\end{document}